\documentclass[a4paper,12pt,notitlepage]{article}

\usepackage[utf8]{inputenc}
\usepackage[T1]{fontenc}
\usepackage[dvips]{graphicx}
\usepackage{amssymb}
\usepackage{amsmath}
\usepackage{enumerate}
\usepackage{amsthm}
\usepackage[top=2cm,bottom=2cm,left=2cm,right=2cm]{geometry}
\usepackage{srcltx}

\newtheorem{theorem}{Theorem}[section]
\newtheorem{prop}[theorem]{Proposition}
\newtheorem{cor}[theorem]{Corollary}

\newtheorem{lem}[theorem]{Lemma}
\newtheorem{dfn}[theorem]{Definition}
\numberwithin{equation}{section}

\def\RR{\mathbb{R}}

\def\Rn{\mathbb{R}^{n}}

\def\Rnuno{\mathbb{R}^{n-1}}

\def\Hnuno{\mathcal{H}^{n-1}}
\def\Hndos{\mathcal{H}^{n-2}}

\begin{document}
\title{Balanced split sets and Hamilton-Jacobi equations}

\author{Pablo Angulo Ardoy\\
Department of Mathematics\\Universidad Autónoma de Madrid. \and 
Luis Guijarro \\
Department of Mathematics\\Universidad Autónoma de Madrid.\\ICMAT CSIC-UAM-UCM-UC3M}
\maketitle

\begin{abstract}
We study the singular set of solutions to Hamilton-Jacobi equations with a Hamiltonian independent of $u$. In a previous paper, we proved that the singular set is what we called a balanced split locus. In this paper, we find and classify all balanced split loci, identifying the cases where the only balanced split locus is the singular locus, and the cases where this does not hold.
% solve completely the problem whether there are balanced split loci other than the singular locus by finding and classifying all balanced split sets. 
This clarifies the relationship between viscosity solutions and the classical approach of characteristics, providing equations for the singular set.
Along the way, we prove more structure results about the singular sets.
\end{abstract}

\section{Introduction}

In this paper we consider the following boundary value problem:
\begin{eqnarray}
H(p,du(p))=1&\quad&p\in \Omega\label{HJequation}\\
u(p)=g(p)&&p\in \partial\Omega\label{HJboundarydata}
\end{eqnarray}
for a smooth compact manifold $\Omega$ of dimension $n$ with boundary, $H$ smooth, $H^{-1}(1)\cap T^{\ast}_{p}\Omega$ strictly convex for every $p$, and $g$ smooth and satisfying the compatibility condition:
\begin{equation}\label{compatibility condition} 
 \vert g(y) - g(z)\vert < d(y,z)\qquad \forall y,z \in\partial\Omega
\end{equation}
where $d$ is the distance induced by the Finsler metric:

\begin{equation}\label{phi is the dual of H}
 \varphi_{p}(v)=\sup\left\lbrace 
\left\langle v,\alpha\right\rangle_{p}\, :\,
\alpha\in T^{\ast}_{p}\Omega, \,H(p,\alpha)=1
\right\rbrace 
\end{equation} 

This definition gives a norm in every tangent space $T_p\Omega$.
Indeed, $H$ is a norm at every tangent space if we make the harmless assumption that $H$ is positively homogeneous of order $1$: $H(p,\lambda \alpha)=\lambda H(p, \alpha)$ for $\lambda>0 $.

A unique \emph{viscosity solution} is given by the Lax-Oleinik formula:
\begin{equation}\label{LaxOleinik} 
u(p)=\inf_{q\in \partial \Omega}
\left\lbrace 
   d(p,q)+ g(q)
\right\rbrace
\end{equation} 

A local \emph{classical} solution can be computed near $\partial \Omega$ following \emph{characteristic} curves, which are geodesics of the metric $\varphi$ starting from a point in $\partial \Omega$ with initial speed given by a vector field on $\partial \Omega$ that 
% we call the characteristic vector field. 
is determined by $H$ and $g$ (see \ref{equation for the characteristic vector field}): if $\gamma:[0,t)\rightarrow \Omega$ is the unique (projected) characteristic from a point $q\in\partial\Omega $ to $p=\gamma(t)$ that does not intersect $Sing$, then $u(p)=g(q)+t$. 
The viscosity solution can be thought of as a way to extend the classical solution to the whole $\Omega$.
% 
% The solution by characteristics cannot be defined in all of $\Omega$ in general, because a point may have several characteristics, that propose different .

Let $Sing$ be the closure of the singular set of the viscosity solution $u$ to the above problem. $Sing$ has a key property: any point in $\Omega\setminus Sing$ can be joined to $\partial \Omega $ by a unique characteristic curve that does not intersect $Sing$. 
A set with this property is said to \emph{split} $\Omega$. 
Once characteristic curves are known, if we replace $Sing$ by any set $S$ that splits $\Omega$, we can still apply the formula in the last paragraph to obtain another function with some resemblance to the viscosity solution (see definition \ref{u associated to S}).

The set $Sing$ has an extra property: it is a \emph{balanced split locus}.
This notion, introduced in \cite{Nosotros} and inspired originally by the paper \cite{Itoh Tanaka 2}, is related to the notion of \emph{semiconcave} functions that is now common in the study of Hamilton-Jacobi equations
(see section \ref{section: setting}).
Our goal in this paper is to determine whether there is a unique balanced split locus. In the cases when this is not true, we also give an interpretation of the multiple balanced split loci.

%We believe that results in this line may open the field for new definitions that extend classical solutions for other PDEs solvable by characteristics. In more general settings, conditions similar to Rankine-Hugoniot may replace our balanced condition to yield similar results.

Finally, we recall that the distance function to the boundary in Riemannian and Finsler geometry is the viscosity solution of a Hamilton-Jacobi equation (\cite{Mantegazza Mennucci}), and the \emph{cut locus} is the closure of the singular set of the distance function to the boundary (\cite{Li Nirenberg}). Thus, our results also apply to cut loci in Finsler geometry.

% objects known as \emph{cut loci} in Riemannian and Finsler geometry can be interpreted as singular sets for Hamilton-Jacobi equations (\cite{Mantegazza Mennucci}), and thus our results also apply to these classic geometric constructions.

\subsection{Outline} 

In section 2 we state our results, give examples, and comment on possible extensions. Section 3 gathers some of the results from the literature we will need, and includes a few new lemmas that we use later. Section 4 contains our proof that the \emph{distance to a balanced split locus} and \emph{distance to the $k$-th conjugate point} are Lipschitz. Section 5 contains the proof of the main theorems, modulo a result that is proved in section 6. This last section also features detailed descriptions of a balanced split set at each of the points in the classification introduced in \cite{Nosotros}.

\subsection{Acknowledgements.} 

The authors express their gratitude to Ireneo Peral, Yanyan Li, Luc Nguyen, Marco Fontelos and Juan Carlos \'{A}lvarez for interesting conversations about this problem. The paper also benefited greatly from the referee's input, and the authors want to thank him for it.  Both authors were partially  supported during the preparation of this work by grants MTM2007-61982 and MTM2008-02686 of the MEC and the MCINN respectively. 

\section{Statement of results.}

% We can restrict the exponential map to the following subset:
% \begin{equation}\label{domain of exp}
% \mathcal{N}=...
% \end{equation}

\subsection{Setting}\label{section: setting}

We study a Hamilton-Jacobi equation given by (\ref{HJequation}) and (\ref{HJboundarydata}) in a $C^{\infty}$ compact manifold with boundary $\Omega$.
% with \emph{compact} boundary $\partial \Omega$. The space $\Omega\cup \partial \Omega$ need not be compact. 
$H$ is smooth and strictly convex in the second argument and the boundary data $g$ is smooth and satisfies (\ref{compatibility condition}).

The solution by characteristics gives the \emph{characteristic vector field} on $\partial \Omega $, which we write as a map $\Gamma:\partial \Omega \rightarrow T\Omega$  that is a section of the projection map $\pi:T\Omega\rightarrow \Omega$ of the tangent to $\Omega$.
The characteristic curves are the integral curves of the geodesic vector field in $T\Omega$ with initial point $\Gamma(z)$ for $z\in \partial \Omega$.
The projected characteristics are the projection to $\Omega$ of the characteristics. The characteristic vector field is smooth and points inwards (see \ref{equation for the characteristic vector field}).

Let $\Phi$ be the geodesic flow in $T\Omega$, and $D(\Phi)$ its domain. We introduce the set $ V\subset \RR \times\partial\Omega$:
\begin{equation}\label{V is ...}
 V=\left\lbrace x=(t,z), z\in \partial \Omega, t\geq 0, (t,\Gamma(z))\in D(\Phi) \right\rbrace
\end{equation}
$V$ has coordinates given by $z\in \partial \Omega$ and $t\in \RR$.
We set $F:V\rightarrow \Omega$ to be the map given by following  the projected characteristic with initial value $\Gamma(z)$ a time $t$:
\begin{equation}\label{definition of F} 
 F(t,z)=\pi(\Phi(t,\Gamma(z)))
\end{equation}
The vector $r$ given as $\frac{\partial}{\partial t}$ in the above coordinates maps under $F$ to the tangent to the projected characteristic.

\begin{dfn}\label{splits}
For a set $S\subset \Omega$, let $A(S)\subset V$ be the set of all $x=(t,z)\in V$ such that $F(s,z)\notin S ,\, \forall \, 0\leq s<t$.
We say that a set $S\subset \Omega$ \emph{splits} $\Omega$ iff $F$ restricts to a bijection between $A(S)$ and $\Omega\setminus S$.
\end{dfn}

Whenever $S$ splits $\Omega$, we can define a vector field $R_{p}$ in $\Omega\setminus S$ to be $dF_{x}(r_{x})$ for the unique $x$ in $A(S)$ such that $F(x)=p$.
\begin{dfn}\label{the set R_a}

For a point $a\in S$, we define the \emph{limit set} $R_{a}$ as the set of vectors in $T_{a}\Omega$ that are limits of sequences of the vectors $R_{p}$ defined above at points $p\in \Omega\setminus S$.

% We define $V^{\ast}_{a}$ as the set of the duals to the vectors in $V_{a}$, and $coV^{\ast}_{a}$ as its convex hull.
\end{dfn}

\begin{dfn}\label{the set Q_p}
If $S$ splits $\Omega$, we also define a set $Q_{p}\subset V$ for $p\in \Omega$ by
$$
Q_{p}=\left( F\vert_{\overline{A(S)}}\right)^{-1}(p)
$$
\end{dfn}
The following relation holds between the sets $R_{p}$ and $Q_{p}$:
$$
R_{p}=\left\lbrace dF_{x}(r_{x}):x\in Q_{p}
\right\rbrace 
$$

\begin{dfn}\label{u associated to S}
If $S$ splits $\Omega$, we can define a real-valued function $h$ in $\Omega\setminus S$ by setting:
$$
h(p)=g(z)+t
$$
where $(t,z)$ is the unique point in $A(S)$ with $F(t,z)=p$.
\end{dfn}

% If $S=Sing$, the above construction returns  the original viscosity solution $u$ from (\ref{LaxOleinik}).

If we start with the viscosity solution $u$ to the Hamilton-Jacobi equations, and let $S=Sing$ be the closure of the set where $u$ is not $C^1$, then $S$ splits $\Omega$. If we follow the above definition involving $A(S)$ to get a new function $h$, then we find $h=u$.

\begin{dfn}\label{split locus}
 A set $S$ that splits $\Omega$ is a split locus iff
$$
S =\overline{
\left\lbrace 
p\in S: \quad \sharp R_{p}\geq 2
\right\rbrace }
$$
\end{dfn}

The role of this condition is to restrict $S$ to its \textit{essential} part. A set that merely splits $\Omega$ could be too big: actually $\Omega$ itself splits $\Omega$. 
The following lemma may clarify this condition.

\begin{lem}\label{characterization of split locus}
A set $S$ that splits $\Omega$ is a split locus if and only if $S$ is closed and it has no proper closed subsets that split $\Omega$.
\end{lem}
\begin{proof}
The only if part is trivial, so we will only prove the other implication. 
Assume $S$ is a split locus and let $S'\subset S$ be a closed set splitting  $\Omega$. Let $q\in S\setminus S'$ a point with $\sharp R_q\geq 2$. 
Since $S'$ is closed, there is a neighborhood of $q$ away from $S'$; so, if $\gamma_1$ is a segment of a geodesic in $\Omega\setminus S'$ joining $\partial\Omega$ with $q$, there is a point $q_1$ in $\gamma_1$ lying beyond $q$.
% 
% As $S'$ splits $\Omega$, there is one geodesic $\gamma_1$ contained in $\Omega\setminus S'$ that extends from $\partial\Omega$ to $q$. 
% By the hypothesis that $S'$ is closed, there are points of $\gamma_1$ beyond $q$ that can be joined to $\partial\Omega$ with a segment of the geodesic $\gamma_1$ contained in $\Omega\setminus S'$.
Furthermore, we can choose the point $q_1$ not lying in $S$, so there is a second geodesic $\gamma_2$ contained in $\Omega\setminus S\subset \Omega\setminus S'$ from $\partial\Omega$ to $q_1$.
As $q\in S$, we see $\gamma_2$ is necessarily different from $\gamma_1$, which is a contradiction if $S'$ is split.
Therefore we learn $S'\supset\left\lbrace p\in S: \quad \sharp R_{p}\geq 2\right\rbrace$, so $S=\overline{\left\lbrace p\in S: \quad \sharp R_{p}\geq 2\right\rbrace } \subset S'$.

\end{proof}

Finally, we introduce the following more restrictive condition (see \ref{vector de x a y} for the definition of $v_p(q)$, the \emph{vector from $p$ to $q$}, and \ref{dual one form} for the Finsler dual of a vector).

\begin{dfn}\label{balanced} 
% In the setting \ref{section: setting}, w
We say a split locus $S\subset \Omega$ is \emph{balanced} for given $\Omega$, $H$ and $g$ (or simply that it is balanced if there is no risk of confusion) iff  for all $p\in S$, 
all sequences $p_{i}\to p$ with $v_{p_{i}}(p)\to v\in T_{p}\Omega$, and any sequence of vectors $X_{i}\in R_{p_{i}}\to X_{\infty}\in R_{p}$, 
then 
$$
w_{\infty}(v)=\max\left\lbrace w(v),\text{ $w$ is dual to some $R\in R_{p}$} \right\rbrace
$$
where $w_{\infty}$ is the dual of $X_{\infty}$.
%We say a split locus $S\subset \Omega$ is \emph{balanced} for given $\Omega$, $H$ and $g$ (or simply that it is balanced if there is no risk of confusion) iff the following condition is satisfied for all $p\in S$:
%
%Let $p_{n}\subset \Omega$ be a sequence converging to $p$ such that $v_{p_{n}}(p)$ converges to $v\in T_{x}\Omega$. Let $X_{n}\in R_{p_{n}}$ be a sequence of vectors that converges to $X_{\infty}\in R_{p}$, 
%and let $w_{\infty}$ be the dual of $X_{\infty}$.
%% and let $w_{n}$ and $w_{\infty}$ be the duals of $V_{n}$ and $V_{\infty}$.
%Then
%
%$$
%w_{\infty}(v)=\max\left\lbrace w(v),\text{ $w$ is dual to some $R\in R_{p}$} \right\rbrace
%$$

\end{dfn}
% \paragraph{Remark.} If a set $S$ satisfies the first and third condition but not the second, then $S$ can be reduced to those points at which $R$ has more than one element.

\paragraph{Remark.} We proved in \cite{Nosotros} that the cut locus of a submanifold in a Finsler manifold and the closure of the singular locus of the solution to (\ref{HJequation}) and (\ref{HJboundarydata}) are always balanced split loci.
The proof (and the definition of balanced itself) was inspired by the paper \cite{Itoh Tanaka 2}, and consists basically of an application of the first variation formula.

We give now another proof that relates the balanced condition to the notion of semiconcave functions, which is now common in the study of Hamilton-Jacobi equations.
More precisely, we simply translate theorem 3.3.15 in the book \cite{Cannarsa Sinestrari} to our language to get the following lemma:

% is the statement, in a different language, that the closure of the singular locus of the viscosity solution to (\ref{HJequation}) and (\ref{HJboundarydata}) are always balanced split loci.

\begin{lem}
The closure of the singular set of the viscosity solution to (\ref{HJequation}) and (\ref{HJboundarydata}) is a balanced split locus.
\end{lem}
\begin{proof}
Let $u$ be the viscosity solution to (\ref{HJequation}) and (\ref{HJboundarydata}),
and let $Sing$ be the closure of its singular set.
We leave to the reader the proof that $Sing$ is a split locus (otherwise, see \cite{Nosotros}).

It is well known that $u$ is semiconcave (see for example \cite[5.3.7]{Cannarsa Sinestrari}).
The superdifferential $ D^+ u(p) $ of $u$ at $p$ is the convex hull of the set of limits of differentials of $u$ at points where $u$ is $C^1$ (see \cite[3.3.6]{Cannarsa Sinestrari}).
At a point where $u$ is $C^1$, the dual of the speed vector of a characteristic is the differential of $u$.
Thus, the superdifferential at $p$ is the convex hull of the duals to the vectors in $R_p$.
We deduce:
$$
\max\left\lbrace w(v),\text{ $w$ is dual to some $R\in R_{p}$} \right\rbrace =
\max\left\lbrace w(v), w\in D^+ u(p) \right\rbrace
$$
Given $p\in \Omega$, and $v\in T_p\Omega$, the \emph{exposed face} of $D^+ u(p) $ in the direction $v$ is given by:
$$
D^+ (p,v) = \{ \tilde{w}\in D^+ u(p) : \tilde{w}(v)\leq w(v) \;\forall w \in D^+ u(p) \}
$$
The balanced condition can be rephrased in these terms as:

\begin{quote}
 Let $p_i\rightarrow p\in S$ be a sequence with $v_{p_{i}}(p)\to v\in T_{p}\Omega$, and let $w_i\in D^+ u(p_{i})$ be a sequence converging to $w\in D^+ u(p)$.
 
 Then $w \in D^+ u(p,-v) $
\end{quote}
which is exactly the statement of theorem \cite[3.3.15]{Cannarsa Sinestrari}, with two minor remarks:
\begin{enumerate}
 \item The condition is restricted to points $p\in S$. At points in $\Omega\setminus S$, the balanced condition is trivial.
 \item In the balanced condition, we use the vectors $v_{p_i}(p)$ from $p_i$ to $p$, contrary to the reference \cite{Cannarsa Sinestrari}. Thus the minus sign in the statement.
\end{enumerate}

% We learn from theorem 3.3.6 in \cite{Cannarsa Sinestrari} that the convex hull of $D^\ast u(x) $ is $ D^+ u(x) $.
% \begin{equation}\label{the subdifferential and the limiting gradients}
%  \partial u(x,\theta) = \min_{p\in D^+ u(x)}  p(\theta) = \min_{p\in D^\ast u(x)}  p(\theta)
% \end{equation} 
% where $\partial u(x,\theta)$ is the directional derivative, $D^+$ is the superdifferential of $u$, $D^\ast$ is the set of limits of differentials of $u$ at points where $u$ is $C^1$.
% In that same theorem, we also learn that the convex hull of $D^\ast u(x) $ is $ D^+ u(x) $.
% 
% At a point where $u$ is $C^1$, the dual of the speed vector of a characteristic is the differential of $u$.
% Thus, the subdifferential at $p$ is the convex hull of the duals to the vectors in $R_p$.
% The directional derivative in \ref{the subdifferential and the limiting gradients} is given by:
% $$
% 
% $$
% the limit of the 

\end{proof}

In the light of this new proof, we can regard the balanced condition as a differential version of the semiconcavity condition. A semiconcave function that is a solution to (\ref{HJequation}) is also a viscosity solution (see \cite[5.3.1]{Cannarsa Sinestrari}). This papers tries to recover the same result under the balanced condition.

% In the light of this new proof, we can regard the balanced condition as a translation of a property of semiconcave functions, carried over to a slightly different setting, where it can be satisfied by functions that are not even continuous.

% The balanced condition reduces to a simpler form at most points of $S$.
% We will see that most points of a balanced split locus have exactly two incoming characteristics.
% The balanced condition corresponds roughly to saying that the speed vectors to the incoming characteristics make the same angle with tangent vectors to $S$, where a tangent vector to $S$ is any $v\in T_{p}\Omega$ that is the limit of vectors $v_{p_{n}}(p)$ for a sequence $p_n\rightarrow p$ with $p_n\in S$(see \cite{Nosotros} for more detail, specially the proof of proposition 7.2).
% 
% At the rest of points, TODO?

\subsection{Results}

For \emph{fixed} $\Omega$, $H$ and $g$ satisfying the conditions stated earlier, there is always at least one balanced split locus, namely the singular set of the solution of (\ref{HJequation}) and (\ref{HJboundarydata}).
In general, there might be more than one balanced split loci, depending on the topology of $\Omega$. 

Our first theorem covers a situation where there is uniqueness.

\begin{theorem}\label{maintheorem0}
Assume $\Omega$ is simply connected and $\partial \Omega$ is connected.

Then there is a unique balanced split locus, which is the singular locus of the solution of (\ref{HJequation}) and (\ref{HJboundarydata}).
\end{theorem}

The next theorem removes the assumption that $\partial \Omega$ is connected, at the price of losing uniqueness:

\begin{theorem}\label{maintheorem1}
Assume $\Omega$ is simply connected and $\partial \Omega$ has several connected components.
Let $S\subset \Omega$ be a balanced split locus.

Then $S$ is the singular locus of the solution of (\ref{HJequation}) and (\ref{HJboundarydata}) with boundary data $g+a$ where the function $a$ is constant at each connected component of $\partial\Omega$.
% whose characteristic vector field is $V$.
\end{theorem}

The above theorem describes precisely all the balanced split loci in a situation where there is non-uniqueness. If $\Omega$ is not simply connected, the balanced split loci are more complicated to describe. We provide a somewhat involved procedure using the universal cover of the manifold. However, the final answer is very natural in light of the examples.

\begin{theorem}
There exists a bijection between balanced split loci for given $\Omega $, $H$ and $g$ and an open subset of the homology space $H^1(\Omega, \partial\Omega)$ containing zero.
\end{theorem}

In fact, this theorem follows immediately from the next, where we construct such bijection:

\begin{theorem}\label{maintheorem2}
Let $\widetilde{\Omega}$ be the universal cover of $\Omega$, and lift both $H$ and $g$ to $\widetilde{\Omega} $.
% Indeed, let $p:\widetilde{\Omega}\rightarrow \Omega$ be the universal cover of $\Omega$, and lift both $H$ and $g$ to $\widetilde{\Omega} $.

Let $a:[\partial \widetilde{\Omega}]\rightarrow \RR $ be an assignment of a constant to each connected component of $\partial \widetilde{\Omega} $ that is equivariant for the action of the automorphism group of the covering and such that $\widetilde{g}(z)+a(z)$ satisfies the compatibility condition (\ref{compatibility condition}) in $\widetilde{\Omega} $. Then 
the singular locus $\widetilde{S}$ of the solution $\widetilde{u} $ to:

$$
\widetilde{H}(x,d\widetilde{u}(x))=1 \quad x\in \widetilde{\Omega}
$$
$$
\widetilde{u}(x)=\widetilde{g}(x)+a(z)  \quad x\in \partial\widetilde{\Omega}
$$
\noindent is invariant by the automorphism group of the covering, and its quotient is a set $S$ that is a balanced split locus for $\Omega $, $H$ and $g$. Furthermore:

\begin{enumerate}
 \item The procedure above yields a bijection between balanced split loci for given $\Omega $, $H$ and $g$ and \emph{equivariant compatible} functions 
$a:[\partial \widetilde{\Omega}]\rightarrow \RR $.
 \item Among the set of equivariant functions $a:[\partial \widetilde{\Omega}]\rightarrow \RR $ (that can be identified naturally with $H^1(\Omega, \partial\Omega)$), those compatible correspond to an open subset of $H^1(\Omega, \partial\Omega)$ that contains $0$.
\end{enumerate}

% Two functions $a$ and $a'$ that differ by a global constant are equivalent and yield the same $S$.
% All balanced split sets come from a unique such function $a$ (up to a global constant).

% More explicitly, there is a correspondence between balanced split loci and functions $a$ as above, defined up to addition of a constant. An open neighborhood of the space of equivariant functions (and thus of the set of balanced split loci) is given by $H^1(\Omega, \partial\Omega)$, which is dual to $ H_{n-1}(\Omega)$ by Lefschetz theorem.
\end{theorem}

\paragraph{Remark.}
The space $H^1(\Omega, \partial\Omega)$ is dual to $ H_{n-1}(\Omega)$ by Lefschetz theorem. The proof of the above theorems rely on the construction from $S$ of a $(n-1)$-dimensional current $T_S$ that is shown to be closed and thus represents a cohomology class in $H_{n-1}(\Omega)$. The proof of the above theorem also shows that the map sending $S$ to the homology class of $T_S$ is a bijection from the set of balanced split loci onto a subset of $H_{n-1}(\Omega)$.

% \begin{theorem}\label{maintheorem2}
% Let $\Omega$ be a manifold and $p:\widetilde{\Omega}\rightarrow \Omega$ be its universal cover.
% 
% Let $a:[\partial \widetilde{\Omega}]\rightarrow \RR $ be an assignment of a constant to each connected component of $\partial \widetilde{\Omega} $ that is equivariant for the action of the automorphism group of the covering and such that $\widetilde{g}(z)=g(p(z))+a(z)$ satisfies the compatibility condition \ref{compatibility condition} in $\widetilde{\Omega} $. Then the singular locus $\widetilde{S}$ of the solution $\widetilde{u} $ to:
% 
% $$
% \widetilde{H}(x,d\widetilde{u}(x))=1 \quad x\in \widetilde{\Omega}
% $$
% $$
% \widetilde{u}(x)=\widetilde{g}(x)  \quad x\in \partial\widetilde{\Omega}
% $$
% 
% is invariant by the automorphism group of the covering, and its quotient is a set $S$ that is a balanced split locus in $\Omega$. Two functions $a$ and $a'$ that differ by a global constant are equivalent and yield the same $S$.
% 
% All balanced split sets can be constructed in this way.
% 
% More explicitly, there is a correspondence between balanced split loci and functions $a$ as above, defined up to addition of a constant. An open neighborhood of the space of equivariant functions (and thus of the set of balanced split loci) is given by $H^1(\Omega, \partial\Omega)$, which is dual to $ H_{n-1}(\Omega)$ by Lefschetz theorem.
% \end{theorem}

In order to prove the above we will make heavy use of some structure results for balanced split loci. To begin with, we use the results of \cite{Nosotros}, which were stated for a balanced split locus with this paper in mind. In the last section, we improve the description of the cut locus near each of these types of points.

We also study some very important functions for the study of the cut locus. Recall the global coordinates in $V$ given by $z\in \partial \Omega $ and $t\in \RR$.
Let $\lambda_j(z)$ be the value of $t$ at which the geodesic $s\rightarrow \Phi (s,z)$ has its $j$-th conjugate point (counting multiplicities), or $\infty$ if there is no such point.
Let $\rho_S:\partial \Omega \rightarrow \RR$ be the minimum $t$ such that $F(t,z)\in S$.

\begin{lem}\label{landa es Lipschitz} 
All functions $\lambda_j: \partial\Omega \rightarrow \RR $ are Lipschitz continuous.
\end{lem}

\begin{lem}\label{rho is lipschitz}
The function $\rho_{S}:\partial \Omega \rightarrow \RR$ is Lipschitz continuous if $S$ is balanced.
\end{lem}

Both results were proven in \cite{Itoh Tanaka 2} for Riemannian manifolds, and the second one was given in \cite{Li Nirenberg}. 
Thus, our results are not new for a cut locus, but the proof is different from the previous ones and may be of interest.
We have recently known of another proof that $\rho$ and $\lambda_1$ are Lipschitz (\cite{Castelpietra Rifford}).

\subsection{Examples}

Take as $\Omega$ any ring in a euclidean $n$-space bounded by two concentric spheres. Solve the Hamilton-Jacobi equations with $H(x,p)=\vert p\vert$ and $g=0$. The solution is the distance to the spheres, and the cut locus is the sphere concentric to the other two and equidistant from each of them. However, any sphere concentric to the other two and lying between them is a balanced split set, so there is a one parameter family of split balanced sets. When $n>2$, this situation is a typical application of \ref{maintheorem1}. In the $n=2$ case, there is also only one free parameter, which is in accord with \ref{maintheorem2}, as the rank of the $H_1$ homology space of the ring is one.

% In the inner sphere, take the unit normal vector field pointing outside the sphere into the ring, and take the unit normal vector field pointing inside the sphere for the outer one. Any sphere concentric to the other two and lying between them is a balanced split set, so there is a one parameter family of split balanced sets. When $n>2$, this situation is a typical application of our theorem when $\partial\Omega $ has more than one connected component. It is not hard to prove that in the $n=2$ case, there is also only one free parameter, despite the fact that $\Omega$ is both non simply-connected and its boundary has several connected components.

\begin{figure}[ht]
 \centering
 \includegraphics{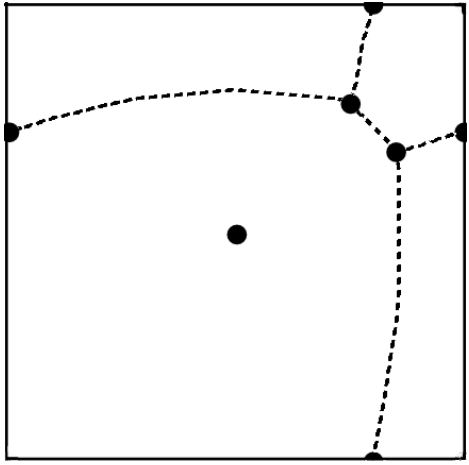}
 % toro.inkscape.eps: 1179666x1179666 pixel, 0dpi, infxinf cm, bb=
 \caption{Balanced split set in a torus}
 \label{fig: Balanced split set in a euclidean torus}
\end{figure}

For a more interesting example, we study balanced split sets with respect to a point in a euclidean torus.
% A cut locus with respect to a point is also the cut locus with respect to a small distance sphere around the point, so that we can take $\Omega$ to be the torus minus a unit ball around one point to fit our setting.
We take as a model the unit square in the euclidean plane with its borders identified.
It is equivalent to study the distance with respect to a point in this euclidean torus, or the solution to Hamilton-Jacobi equations with respect to a small distance sphere centered at the point with the Hamiltonian $H(p)=\vert p\vert$ and $g=0$.

A branch of cleave points (see \ref{structure up to codimension 3}) must keep constant the difference of the distances from either sides (recall prop 7.2 in \cite{Nosotros}, or read the beginning of section \ref{section:proof of the main theorems}). Moving to the covering plane of the torus, we see they must be segments of hyperbolas. A balanced split locus is the union of the cleave segments and a few triple or quadruple points. The set of all balanced split loci is a $2$-parameter family, as predicted by our theorem \ref{maintheorem2}.

\subsection{Extensions}
The techniques in this paper could be applied to other first order PDEs, or systems of PDEs.
In particular, we can mention the Cauchy problem with a Hamiltonian dependent on $t$, and both Cauchy and Dirichlet problems with a Hamiltonian dependent on $u$.
Characteristic curves are well behaved on those cases (though some extra hypothesis are needed for a Hamiltonian dependent on $u$).

In particular, we believe our proofs of \ref{landa es Lipschitz} and \ref{rho is lipschitz} are more easily extensible to other settings than the previous ones in the literature. 
This may simplify the task of proving that the singular locus for other PDEs have locally finite $n-1$ Hausdorff measure.

In this paper and its predecessor \cite{Nosotros} strong regularity assumptions were assumed. There are powerful reasons to weaken the regularity assumptions when studying Hamilton-Jacobi equations. The definition of a balanced split locus itself does not require strong regularity. Less regular data, though, could produce qualitatively different behavior. 
In the structure result \ref{structure up to codimension 3}, the dimensions of the sets of points of each type may become higher, as a consequence of the general Morse-Sard-Federer theorem (see \cite{Federer}).
Also, if $g$ is not $C^1$, we can expect non-trivial intersections between the singular set and $\partial \Omega$, or \emph{rarefaction waves}.
\section{Preliminaries}

\subsection{Definitions}
\begin{dfn}\label{riemannian metric at (x,v)} 
Let $v\in T_{p}\Omega$ be a tangent vector at $p$ in a Finsler manifold $(\Omega, \varphi)$. The \emph{Riemannian metric} at $(p,v)$ is given by:
$$
g_{(p,v)}(X,Y)=\frac{\partial^2\varphi}{\partial v^{i}v^{j}}(p,v)X^i Y^j
$$
\end{dfn}

\begin{dfn}\label{dual one form}
The \emph{dual one form} to a vector $X\in T_{p}\Omega$ with respect to a Finsler metric $\varphi $ is the unique one form $w \in T^{\ast}_{p}\Omega$ such that $w(X)=\varphi(X)^{2}$ and $w\vert_{H}=0$, where $H$ is the hyperplane tangent to the level set 
$$
\left\lbrace Y \in T_{p}\Omega, \varphi(Y) = \varphi(X) \right\rbrace 
$$
at $X$. It coincides with the usual definition of dual one form in Riemannian geometry.

For a vector field, the dual differential one-form is obtained by applying the above construction at every point.

We will often use the notation $\widehat{X}$ for the dual one-form to the vector $X$.
\end{dfn}

In coordinates, the dual one form $w$ to the vector $X$ is given by:
$$
w_{j}=\frac{\partial\varphi}{\partial v^{j}}(p,X)
$$
and also, in terms of the Riemannian metric at $(p,X)$:
$$
w(\cdot)=g_{(p,X)}(X,\cdot)
$$

With this notion of dual form, we can restate the usual equations for the characteristic vector field at points $p\in \partial\Omega $:
\begin{eqnarray}\label{equation for the characteristic vector field} 
 &\varphi_p(X_p)=1\notag\\
 &\widehat{X_p}\vert_{T(\partial\Omega)}=d g\notag\\
 &X_p\text{ points inwards}
\end{eqnarray}

\begin{dfn}\label{vector de x a y}
Whenever there is a unique unit speed minimizing geodesic $\gamma$ joining the points $p$ and $q$ in $\Omega$, we define, following \cite{Itoh Tanaka 2}, 
\begin{equation}
 v_{p}(q)=\dot{\gamma}(0)
\end{equation} 
For fixed $p\in  \Omega $, then any $q$ sufficiently close to $p$ is joined to $p$ by a unique unit speed minimizing geodesic, so $v_p(q)$ is well defined.

\end{dfn}

% ¿hace falta el tangent cone? creo que no
% \begin{dfn}\label{approximate tangent cone}
% The \emph{approximate tangent cone} to a subset $E$ at $x$ is:
% $$T(E,x)=\left\lbrace
% r\theta,\theta=\lim v_{x}(x_{n}),
% \exists \lbrace x_{n}\rbrace\subset E, x_{n}\rightarrow x
% \right\rbrace $$
% and the \emph{approximate tangent space} $Tan(E,x)$ to $E$ at $x$ is the vector space generated by $T(E,x)$.
% \end{dfn}

\begin{dfn}\label{conjugate}
Let $z\in \partial \Omega$ and $x=(t,z)\in V$.

We say $x$ is \emph{conjugate} iff $F$ is not a local diffeomorphism at $x$. The \emph{order} of conjugacy is the dimension of the kernel of $dF$.

We say $x$ is \emph{a first conjugate vector} iff no point $(s,x)$ for $s<t$ is conjugate.

% Let $p\in \partial\Omega$ and $V\in T_{p}\Omega$ be a vector in the set \ref{domain of exp}. Let $F$ be a local ? of the distribution orthogonal  to $V$.
% We say $V$ is \emph{conjugate} iff the restriction of the exponential map to the normal bundle of $F$ is not a local diffeomorphism at $V$.
\end{dfn}

We recall from \cite{Nosotros} a result on the structure of balanced split loci (in that paper, conjugate points are called focal points):

\begin{theorem}\label{structure up to codimension 3} 
A balanced split locus consists of the following types of points:

 \begin{itemize}
 \item\textbf{Cleave points:} Points at which $R_{p}$ consists of two non-conjugate vectors. The set of cleave points is a smooth hypersurface;
 \item \textbf{Edge points:} Points at which $R_{p}$ consists of exactly one conjugate vector of order 1. This is a set of Hausdorff dimension at most $n-2$;
 \item\textbf{Degenerate cleave points:} Points at which $R_{p}$ consists of two vectors, such that one of them is conjugate of order 1, and the other may be non-conjugate or conjugate of order 1. This is a set of Hausdorff dimension at most $n-2$;
 \item\textbf{Crossing points:} Points at which $R_{p}$ consists of non-conjugate and conjugate vectors of order 1, and $R^{\ast}_{p} $ spans an affine subspace of dimension $2$ ($R^{\ast}_{p} $ is the set of duals to vectors in $R_{p} $). This is a rectifiable set of dimension $n-2$;
 \item \textbf{Remainder:} A set of Hausdorff dimension at most $n-3$;
\end{itemize}
\end{theorem}

% \subsection{Gauss lemma}
% 
% ...

\subsection{Special coordinates}\label{subsection: special coordinates}
In \cite{Nosotros}, we used only a few properties of the exponential map essentially introduced in \cite{Warner}. Those properties, stated in proposition 8.3 of \cite{Nosotros}, were shown enough to guarantee the existence of special coordinates for $F$ near a conjugate point of order $k$ (see the paragraph on special coordinates before theorem 6.3 of \cite{Nosotros}). Near a point $x^0\in V$ and its image $F(x^0)\in \Omega$, we can find coordinates such that $x^0$ has coordinates $0$, and $F$ is written as
\begin{equation}\label{F near order k in special coords} 
 F(x_{1},\dots,x_{n})=(x_{1},\dots,x_{n-k},F_{n-k+1}(x),\dots,F_{n}(x))
\end{equation} 
where 
\begin{itemize}
\item $\frac{\partial}{\partial x_i} F_j(x^0)$ is $0$ for any $i$ and $n-k+1\leq j\leq n$, 
\item $\frac{\partial}{\partial x_i} \frac{\partial}{\partial x_1} F_j(x^0)$ is $\delta^i_j$, for $n-k+1\leq i\leq n$ and $n-k+1\leq j\leq n$.
\item $\frac{\partial}{\partial x_1} (x^0) = r_{x^0} $
\end{itemize}
% ???
% 
% We recall briefly how these coordinates were constructed, but will introduce a modification that will be useful later. (¿estás seguro?)
% There is a basis 
% \begin{equation}
%  \label{first basis for warner R2}
% \mathcal{B}=\left\lbrace v_{1},..,v_{n} \right\rbrace  
% \end{equation}
% 
% of $T_{x}V$ where $r=v_{1}$ and $v_{n-k+1},..,v_{n}$ span $\ker dF_{x}$, and such that:
% \begin{equation}
%  \label{second basis for warner R2}
% \mathcal{B}'=
% \left\lbrace 
% dF(v_{1}),\dots,dF(v_{n-k}),
% \widetilde{d^{2}F(r\sharp v_{n-k+1})},\dots,
% \widetilde{d^{2}F(r\sharp v_{n})},
% \right\rbrace 
% \end{equation}
% is a basis of $T_{F(x)}M$, where $\widetilde{d^{2}F(r\sharp v_{i})}$ is a representative of $d^{2}F(r\sharp v_{i})\in T_{F(x)}M/dF(TV_{x}) $.
% After a linear coordinate change, assume these are the respective canonical basis.
% Near $x$, we keep for the moment the  coordinates in $V$, and near $p$ we introduce the coordinates
% $$
% \phi(p)=(t,v)\in\RR\times L, p=\gamma(t)+v
% $$
% where $\gamma(t)$ is the radial geodesic through $x$ and $L$ is a transversal hyperplane (this is the difference with the coordinates used in \cite{Nosotros}).
% 
% Finally, we notice that the functions $F^i=\phi^i\circ F$ for $i=1,\dots,n-k$, together with the $\phi^i$ for $i=n-k+1,\dots,n$, are a coordinate system near $x$ which together with the previous one near $p$ has the desired properties.
% 
% We remark that in these coordinates the first coordinate vector field is different from the vector field $r$ defined earlier, though they coincide at $x$.

\subsection{Lagrangian submanifolds of $T^\ast \Omega$}\label{subsection: lagrangian}
Let $D$ be the homeomorphism between $T \Omega$ and $T^\ast \Omega$ induced by the Finsler metric as in definition \ref{dual one form} ($D$ is actually a $C^{\infty}$ diffeomorphism away from the zero section). We define a map:
\begin{equation}\label{definition of the embedding of V into TastOmega} 
\Delta(t,z) = D(\Phi(t,\Gamma(z)))
\end{equation} 
and a subset of $T^\ast\Omega$:
\begin{equation} \label{definition of Theta} 
\Theta = \Delta(V)
% \{ D(\Phi(t,\Gamma(z))): (t,z)\in V\} 
\end{equation} 
where $\Phi $ is the geodesic flow in $T\Omega$. This is a smooth $n$-submanifold of $T^\ast \Omega$ with boundary.

% The duality between vectors and 1-forms given by the Finsler metric gives a diffeomorphism between $T \Omega$ and $T^\ast \Omega$. 
It is a standard fact that, for a function $u:\Omega\rightarrow \RR$, the graph of its differential $du$ is a Lagrangian submanifold of $T^\ast \Omega$. The subset of $\Theta$ corresponding to small $t$  is the graph of the differential of the solution $u$ to the HJ equations by characteristics. Indeed, all of $\Theta$ is a lagrangian submanifold of $T^\ast \Omega$ (see \cite{Duistermaat}).

We can also carry over the geodesic vector field from $T \Omega$ into $T^\ast \Omega$ (outside the zero sections). This vector field in $T^\ast \Omega$ is tangent to $\Theta$. Then, as we follow an integral curve $\gamma(t)$ within $\Theta$, the tangent space to $\Theta$ describes a curve $\lambda(t)$ in the bundle $G$ of lagrangian subspaces of $T^\ast \Omega$.  It is a standard fact that the vector subspace $\lambda(t)\subset T^\ast_{\gamma(t)} \Omega $ intersects the vertical subspace of $T^\ast_{\gamma(t)}\Omega$ in a non-trivial subspace for a discrete set of times. We will review this fact, in elementary terms, and prove a lemma that will be important for the proof of lemma \ref{rho is lipschitz}.

Let  $ \eta(t)$ be an integral curve of $r$ with $x_0=\eta(0)$ a conjugate point of order $k$. In special coordinates near $x_0$, for $t$ close to $0$, the differential of $F$ along $\eta$ has the form:
$$
dF(\eta(t))=
\begin{pmatrix}
I_{n-k} & 0 \\ \ast & \ast
\end{pmatrix}
=\begin{pmatrix}
I_{n-k} & 0 \\ 0 & 0
\end{pmatrix}
+t\begin{pmatrix}
0 & 0 \\ 0 & I_k
\end{pmatrix}
+\begin{pmatrix}
0 & 0 \\ \ast & R(t)
\end{pmatrix}+\begin{pmatrix}
0 & 0 \\ \ast & E(t)
\end{pmatrix}
$$
where $\vert R'(t)\vert <\varepsilon $ and $\vert E\vert <\varepsilon$, with $E=0$ if $\gamma(0)=x_0$.

Let $w\in\ker dF(\eta(t_1))$ and $v\in\ker dF(\eta(t_2))$ be unit vectors in the kernel of $dF$ for $t_1<t_2 $ close to $0$. 
It follows that both $v$ and $w$ are spanned by the last $k$ coordinates.
We then find:
$$
0 = w \cdot dF(\eta(t_2))\cdot v-v \cdot dF(\eta(t_1))\cdot w = 
(t_2-t_1)w\cdot v + w(R(t_2)-R(t_1))v + w(E(t_2)-E(t_1))v
$$
and it follows (for some $t_1<t^\ast<t_2 $):
$$
(t_2-t_1)w\cdot v < |w| |v| (|R'(t^\ast) |+2\varepsilon) (t_2-t_1) < 3\varepsilon (t_2-t_1)
$$
or
\begin{equation}\label{almost orthogonal basis} 
 w\cdot v < 3\varepsilon
\end{equation} 
This also shows that the set of $t$'s such that $dF(\eta(t))$ is singular is discrete.

Say the point $x_0=(z_0,t_0)$ is the $j$-th conjugate point along the integral curve of $r$ through $x_0$ from $z_0$, and recall that it is of order $k$ as conjugate point. As $z$ moves towards $z_0$, all functions $\lambda_j(z),\dots,\lambda_{j+k}(z)$ converge to $t_0 $. Let $z_i$ be a sequence of points converging to $z_0$ such that the integral curve through $z_i$ meets its $k$ conjugate points near $z_0$ at $M$ linear subspaces (e.g. $\lambda_j(z_i)=\dots=\lambda_{j+k_1}(z_i)$; $\lambda_{j+k_1+1}(z_i)=\dots=\lambda_{j+k_2}(z_i)$; ...; $\lambda_{j+k_{M-1}+1}(z_i)=\dots=\lambda_{j+k_M}(z_i)$). we get the following theorem (see also lemma 1.1 in \cite{Itoh Tanaka 2}):
\begin{lem}
The subspaces $\ker d_{(\lambda_{j+k_l}(z_i),z_i)}F$ for $l=1,\dots, M$ converge to orthogonal subspaces of $\ker d_{(\lambda_{j}(z_0),z_0)}F $, for the standard inner product in the special coordinates at the point $(\lambda_{j}(z_0),z_0) $.
\end{lem}

\subsection{A useful lemma}
% We start with a useful abstract lemma.

\begin{lem}\label{the set and the cone}
Let $U$ be an open set in $\Rn$, $A\subset U$ a proper open set, $C^+$ an open cone, $V\subset U$ an arbitrary open set and $\varepsilon>0$ such that at any point $q\in \partial A\cap V$, we have $(q+C^+)\cap (q+B_{\varepsilon}) \subset A$.

Then $\partial A\cap V$ is a Lipschitz hypersurface. Moreover, for any vector $X\in C^+$, take coordinates so that $X=\frac{\partial}{\partial x_{1}}$. Then $\partial A\cap V$ is a graph $S=\lbrace (h(x_{2},..,x_{n}),x_{2},..,x_{n})\rbrace$ for a Lipschitz function $h$.
\end{lem}

\begin{proof}
Choose the vector $X\in C^+$ and coordinate system in the statement. Assume $X$ has norm $1$, so that $q+t X\in q+B_t$ for small positive $t$. Take any point $p\in \partial A\cap V$. We notice that all points $p+t\frac{\partial}{\partial x_{1}}$ for $0<t<\varepsilon$ belong to $A$, and all points $p+t\frac{\partial}{\partial x_{1}}$ for $-\varepsilon<t<0$ belongs to $U\setminus A$.
Indeed, there cannot be a point $p+t\frac{\partial}{\partial x_{1}}\in A$ for $-\varepsilon<t<0$ because the set $(p+ t\frac{\partial}{\partial x_{1}})+ (C^+\cap B_{\varepsilon})$ contains an open neighborhood of $p$, which contains points not in $A$. 
%And if we assume there are points with $t\leq \varepsilon$ we can let $t^{\ast}$ be the supremum of all such points to get a similar contradiction. 
In particular, there is at most one point of $\partial A\cap V$ in each line with direction vector $\frac{\partial}{\partial x_{1}}$.

Take two points $q_{1},q_{2}\in \Rnuno$ sufficiently close and consider the lines $L_{1}=\lbrace (t,q_{1}), t\in \RR \rbrace$ and $L_{2}=\lbrace (t,q_{2}), t\in \RR\rbrace$. Assume there is a $t_{1}$ such that $(t_{1},q_{1})$ belongs to $\partial A$.
If there is no point of $\partial A$ in $L_{2}$ then either all points of $L_{2}$ belong to $A$ or they belong to $U\setminus A$.
Both of these options lead to a contradiction if $((t_{1},q_{1})+C^+)\cap ((t_{1},q_{1})+ B_{\varepsilon})\cap L_{2} \neq \emptyset$ (this condition is equivalent to $K\vert q_{1}-q_{2}\vert <\varepsilon$ for a constant $K$ that depends on $C^+$ and the choice of $X\in C^+$ and the coordinate system).

Thus there is a point $(t_{2},q_{2})\in \partial A$. 
For the constant $K$ above and $t\geq t_{1}+K\vert q_{1}-q_{2}\vert$, the point $(t,q_{2})$ lies in the set $(t_{1},q_{1})+C^+$,
% , and for $t - (t_{1}+K\vert q_{1}-q_{2})\vert\ll 1$, we also have $(q_{2},t)\in B_{\varepsilon}$
% . As there cannot be points of $U\setminus A$ in $(q_{1},t_{1})+C$,
so we have
$$
t_{2}<t_{1}+K\vert q_{1}-q_{2}\vert
$$
The points $q_{1}$ and $q_{2}$ are arbitrary, and the lemma follows.
% Now fix a point $p\in \partial A$ and consider a coordinate line $L=\lbrace q+t\frac{\partial}{\partial x_{n}} \rbrace$, for $q$ sufficiently close to $p$. The line $L$ meets the set $p+C\cap B_{\varepsilon}$.
% %  and $p+(-C)\cap B_{\varepsilon}$.
% There cannot be points in $L\cap $
\end{proof}

\subsection{Some generalities on HJ equations.}

\begin{lem}
 For fixed $\Omega$ and $H$, two functions $g,g':\partial\Omega\rightarrow \RR$ have the same characteristic vector field in $\partial\Omega$ iff $g'$ can be obtained from $g$ by addition of a constant at each connected component of $\partial\Omega$.
\end{lem}
\begin{proof} It follows from (\ref{equation for the characteristic vector field}) that $g$ and $g'$ have the same characteristic vector field at all points if and only $d g=d g'$ at all points.
\end{proof}

For our next definition, observe that given $\Omega$, $H$ and $g$, we can define a map $\tilde{u}: V\rightarrow\RR $ by 
$\tilde{u} (t,z)=t+g(z)$.
\begin{dfn}\label{made of characteristics} 
We say that a function $u: \Omega\rightarrow\RR$ is \emph{made from characteristics} iff $u\vert_{\partial\Omega}=g $ and $u$ can be written as $u(p)=\tilde{u}\circ s$ for a (not necessarily discontinuous) section $s$ of $F:V\rightarrow \Omega $.

% Given $\Omega$, $H$ and $g$, a function $u: \Omega\rightarrow\RR$ is said to be \emph{made of characteristics} iff for every point $x\in\Omega $ there is a characteristic curve $\gamma$ from $ y\in \partial \Omega $ to $x$ such that all values of $u$ for points in $\gamma$ can be computed using the method of characteristics.
\end{dfn}

\paragraph{Remark.} In the paper \cite{Menucci}, the same idea is expressed in different terms: all characteristics are used to build a multi-valued solution, and then some criterion is used to select a one-valued solution. The criterion used there is to select the characteristic with the minimum value of $\tilde{u}$.

\begin{lem}\label{una unica solucion continua por caracteristicas}
The viscosity solution to (\ref{HJequation}) and (\ref{HJboundarydata}) is the unique continuous function that is made from characteristics.% (and agrees with $g$ on $ \partial \Omega$ ?).
\end{lem}
\begin{proof} 
Let $h$ be a function made from characteristics, and $u$ be the function given by formula (\ref{LaxOleinik}). Let $Sing$ be the closure of the singular set of $u$.

% Let $h$ be a continuous function made from characteristics. We show that $h$ coincides with the function $u$ given by formula (\ref{LaxOleinik}). Let $Sing$ be the closure of the singular set of $u$.

Take a point $z\in \partial \Omega$.
% Take any point $z\in \partial \Omega$ such that $h$ is continuous at $F(t,z)$ for all $t<\lambda_1(z)$.
Define:
$$t^\ast_z=\sup
\left\{  
t\geq 0:\;
h(F(\tau,z))=u(F(\tau,z))\;
\forall 0\leq \tau<t
\right\}
$$
% Let $t^\ast_z$ be the supremum of times $t\geq 0$ such that $h$ agrees with $u$ on $F(\tau,z) \;\forall 0\leq\tau<t $.

\textbf{Claim}: $t^\ast_z< \rho_{Sing}(z)$ implies $h$ is discontinuous at $F(t^\ast_z,z)$.

\textbf{Proof of the claim}: 
Assume that $t^\ast_z< \rho_{Sing}(z)$ and $h$ is continuous at $F(t^\ast_z,z)$ for some $z\in \partial\Omega $.

% We see $h(F(t^\ast_z,z))=u(F(t^\ast_z,z)) $. 

% Let $A=A(Sing)$ be the set in definition \ref{splits}. 
% The map $F$ restricts to a diffeomorphism from $A$ onto $\Omega\setminus Sing $.
As $t^\ast_z< \rho_{Sing}(z) <\lambda_1(z) $, there is an open neighborhood $O$ of $(t^\ast_z,z) $ such that $F|_{O}$ is a diffeomorphism onto a neighborhood of $p=F(t^\ast_z,z)$.

By hypothesis, there is a sequence $t_n\rightarrow t^\ast_z$ and $p_n=F(t_n,z)$ such that $h(p_n)\neq u(p_n)$.
As $h$ is built from characteristics using a section $s$, we have $h(p_n)=\tilde{u}(s(p_n))=\tilde{u}((s_n,y_n))=s_n+g(y_n)$, for $(s_n,y_n)\neq (t_n,z)$.

For $n$ big enough, the point $(s_n,y_n) $ does not belong to $O$, as $ (t_n,z)$ is the only preimage of $p_n$ in $O$.
As $h(p_n)\rightarrow h(p) $, and $\partial\Omega$ is compact, we deduce the $s_n$ are bounded. 
We can take a subsequence of $(s_n,y_n) $ converging to $(s_\infty,y_\infty)\not\in O $. So we have $p=F(t^\ast_z,z)=F(s_\infty,y_\infty)$.
If $p\not\in Sing $, we deduce that $\lim_{n\rightarrow \infty} h(p_n)=\tilde{u}(s_\infty,y_\infty) > h(p)=u(p)=\tilde{u}(t^\ast_z,z)$, so $h$ is discontinuous at $p$.

Using the claim, we conclude the proof:
if $h$ is continuous, then $ \rho_{Sing}(z)\leq t^\ast_z $ for all $z\in \partial\Omega $, and $u=h$, as any point in $\Omega$ can be expressed as $F(t,z)$ for some $z$, and some $t\leq\rho_{Sing}(z) $.

\end{proof}

We will need later the following version of the same principle:
\begin{lem}\label{una unica solucion por caracteristicas continua en un conjunto grande}
Let $S$ be a split locus, and $h$ be the function associated to $S$ as in definition \ref{u associated to S}.
If $\rho_S$ is continuous, and $h$ can be extended to $\Omega$ so that it is continuous except for a set of null $\mathcal{H}^{n-1}$ measure, then $S=Sing$.
\end{lem}
\begin{proof} 
Define
$$Y_0=\left\{
z\in\partial\Omega:\; h(F(t,z))\neq u(F(t,z))\;\text{ for some } t\in[0,\rho_{Sing}(z))
\right\} $$
By the claim in the previous lemma, $Y_0$ is contained in:
$$Y=\left\{
z\in\partial\Omega:\; h  \text{ discontinuous at } F(t,z)\; \text{ for some }  t\in[0,\rho_{Sing}(z))
\right\} $$

Let $A=A(Sing)$ be the set in definition \ref{splits}. 
The map $F$ restricts to a diffeomorphism from $A$ onto $\Omega\setminus Sing $.
The set $Y$ can be expressed as:
$$
Y=\pi_2\circ (F\vert_A)^{-1}\left(\{p\in\Omega\setminus Sing:\; h \text{ discontinuous at } p \})\right)
$$
and thus by the hypothesis has null $\Hnuno$ measure.
Therefore, $\partial\Omega\setminus Y_0 $ is dense in $\partial\Omega$.

We claim now that $S\subset Sing $.
To see this, let $p\in S\setminus Sing $. Then $p=F(t^\ast,z^\ast)$ for a unique $(t^\ast,z^\ast) \in A $.
It follows $\rho_S(z^\ast)\leq t^\ast<\rho_{Sing}(z^\ast)$.
As $\rho_S$ is continuous, $\rho_S(z)<\rho_{Sing}(z) $ holds for all $z$ in a neighborhood of $z^\ast $ in $\partial\Omega $ and, in particular, for some $z\in \partial\Omega\setminus Y_0$.
This is a contradiction because, for $\rho_{S}(z)<t<\rho_{Sing}(z) $, $h(F(t,z))=\tilde{u}(t',z') $ for $(t',z')\neq (t,z) $, and $t<\rho_{Sing}(z) $ implies 
$h(F(t,z))=\tilde{u}(t',z')> \tilde{u}(t,z)=u(F(t,z))$, forcing $z\in  Y_0$.

We deduce $S=Sing$ using lemma \ref{characterization of split locus} and the fact that $Sing$ is a split locus. 
\end{proof}

\section{$\rho_S$ is Lipschitz}\label{section:rho is Lipschitz} 

In this section we study the functions $\rho_S$ and $\lambda_j$ defined earlier. The fact that $\rho_S$ is Lipschitz will be of great importance later.
The definitions and the general approach in this section follow \cite{Itoh Tanaka 2}, but our proofs are shorter, provide no precise quantitative bounds,
use no constructions from Riemannian or Finsler geometry, 
% depend only on the properties of an exponential map 
% % given in ??,%\ref{regular exponential map},
% given in the appendix of \cite{Nosotros} (which were inspired by the definition of \emph{regular exponential map} in \cite{Warner})
and work for Finsler manifolds, thus providing a new and shorter proof for the main result in \cite{Li Nirenberg}. The proof that $\lambda_j$ are Lipschitz functions was new for Finsler manifolds when we published the first version of the preprint of this paper. Since then, another preprint has appeared which shows that $\lambda_1$ is actually semi-concave.
%, even though a proof is essentially contained in \cite{Li Nirenberg}.
% Our proof that $\lambda_j$ are Lipschitz functions do not use machinery from Finsler geometry, but only the special coordinates \ref{F near order k in special coords}.

\begin{proof}[Proof of \ref{landa es Lipschitz}]
 It is immediate to see that the functions $\lambda_j $ are continuous, since this is 
property (R3) of Warner (see  \cite[ pp. 577-578 and Theorem 4.5 ]{Warner}).%, therefore we will only prove that they are Lipschitz. 

Near a conjugate point $x^0$ of order $k$, we can take special coordinates as in \ref{subsection: special coordinates}:
% so that $x^0$ has coordinates $0$ and  $F$ is locally:
$$
F(x_1,\dots,x_n)=(x_1,\dots,x_{n-k},F_{n-k+1},\dots,F_n)
$$
% where $\frac{\partial}{\partial x_i} F_j$ is $0$ for any $i$ and $n-k+1\leq j\leq n$ and $\frac{\partial}{\partial x_i} \frac{\partial}{\partial x_1} F_j$ is $\delta^i_j$, for $n-k+1\leq i\leq n$ and $n-k+1\leq j\leq n$ .
Conjugate points near $x$ are the solutions of
$$
d(x_1,\dots,x_n)=det(dF)=\sum_{\sigma}(-1)^\sigma \frac{\partial F_{\sigma(n-k+1)}}{\partial x_{n-k+1}}\dots \frac{\partial F_{\sigma(n)}}{\partial x_n}=0
$$
% At the point $x^0$, or $x_1=\dots=x_n=0$, we have:
% $$
% F_j=\frac{\partial F_j}{\partial x_i}=0
% $$
% for any $1\leq i\leq n$ and any $n-k+1\leq j\leq n$, and also:
% $$
% \frac{\partial^2 F_j}{\partial x_ix_1}=\delta^i_j
% $$
% for any $n-k+1\leq i,j\leq n$. 
>From the properties of the special coordinates, we deduce that:
\begin{equation}\label{d alpha d = 0 for some alphas} 
D^\alpha d(0)=0\qquad
\forall \vert \alpha\vert<k
\end{equation} 
and 
$$
\frac{\partial^k}{\partial x_1^k}d=1
$$

We can use the preparation theorem of Malgrange (see \cite{Golubitsky Guillemin}) to find real valued functions $q$ and $l_i$ in an open neighborhood $U$ of $x$ such that $q(x)\neq 0$ and:
$$
q(x_1,\dots,x_n)d(x_1,\dots,x_n)=x_1^k+x_1^{k-1}l_1(x_2,\dots,x_n)+\dots+l_k(x_2,\dots,x_n)
$$
and we deduce from (\ref{d alpha d = 0 for some alphas}) that
\begin{equation}\label{the lower derivatives of l_i vanish}
 D^\alpha l_i(0)=0\qquad
\forall \vert \alpha\vert<i
\end{equation} 
which implies
\begin{equation}\label{bounds for l_i}
|l_i(x_2,\dots, x_n)| < \bar{C} \max\{\vert x_2\vert, \dots, \vert x_n\vert\}^{i}
\end{equation} 

At any conjugate point $(x_1,\dots,x_n)$, we have $q(x)=0$, so:

\begin{equation*}
 -x_1^k=x_1^{k-1}l_1(x_2,\dots,x_n)+\dots+l_k(x_2,\dots,x_n)\\ 
\end{equation*}
\noindent and therefore
\begin{equation*}
|x_1|^k<|x_1|^{k-1}|l_1|+\dots+|l_k|
\end{equation*}

Combining this and (\ref{bounds for l_i}), we get an inequality for $\vert x_1\vert $ at any conjugate point $(x_1,\dots,x_n)$, where the constant $C$ ultimately depends on bounds for the first few derivatives of $F$:
\begin{equation}
 \vert x_1\vert^k<C\max\{\vert x_1\vert, \dots, \vert x_n\vert\}^{k-1}\max\{\vert x_2\vert, \dots, \vert x_n\vert\}
\end{equation}
We notice that 
$\vert x_1\vert>\max\{\vert x_2\vert, \dots, \vert x_n\vert\}$ 
implies
$\vert x_1\vert^k<C\vert x_1\vert^{k-1}\max\{\vert x_2\vert, \dots, \vert x_n\vert\}$.
In other words:
$$
\vert x_1\vert<\max\{C,1\}\max\{\vert x_2\vert, \dots, \vert x_n\vert\}
$$
This is the statement that all conjugate points near $x$ lie in a cone of fixed width containing the hyperplane $x_1=0$. Thus all functions $\lambda_j$ to $\lambda_{j+k}$ are Lipschitz at $(x_2\dots,x_n)$ with a constant independent of $x$.
\end{proof}
% \begin{proof}
% Near a conjugate point $x$ of order $k$, we can take coordinates so that $F$ is locally:
% $$
% F(x_1,\dots,x_n)=(x_1,\dots,x_{n-k},x_1x_{n-k+1}+R_{n-k+1},\dots,x_1x_n+R_n)
% $$
% thus conjugate points near $x$ are the solutions of
% $$
% 0=det(dF)=x_1^n+x_1^{n-1}l_1+\dots+l_n
% $$
% ...
% \end{proof}

\paragraph{Remark.} A proof of lemma \ref{landa es Lipschitz} in the lines of section \ref{subsection: lagrangian} seems possible: let $\Lambda(\Omega)$ be the bundle of Lagrangian submanifolds of the symplectic linear spaces $T^{\ast}_p \Omega $ and let $\Sigma(\Omega)$ be the union of the Maslov cycles within each $\Lambda_p(\Omega)$.
Define $\lambda:V\rightarrow \Lambda(\Omega)$ where $\lambda(x)$ is the tangent to $\Theta $ at $D(\Phi(x)) $ (recall \ref{definition of Theta}).
The graphs of the functions $\lambda_k$ are the preimage of the Maslov cycle $\Sigma(\Omega)$. The geodesic vector field (transported to $T^{\ast} \Omega $), is transversal to the Maslov cycle. With some effort, the angle (in an arbitrary metric) between this vector field and the Maslov cycle at points of intersection can be bounded from below. This is sufficient to show that the $\lambda_k$ are Lipschitz.

\begin{lem}\label{rho es menor que landa1} 
 For any split locus $S$ and point $y\in \partial\Omega $, there are no conjugate points in the curve $t\rightarrow \exp(ty) $ for $t<\rho_{S}(y) $.
In other words, $\rho_S\leq \lambda_1 $.
\end{lem}
\begin{proof}
Assume there is $x$ with $\rho_S(x)-\varepsilon>\lambda_1(x) $.
By \cite[3.4]{Warner}, the map $F$ is not injective in any neighborhood of $(x,t)$. There are points $(x_n,t_n)$ of $S$ with $x_n\rightarrow x$ and $t_n<\rho_S(x)-\varepsilon$ (otherwise $S$ does not split $\Omega $). Taking limits, we see $F(x,t)$ is in $S$ for some $t<\rho_S(x)-\varepsilon$, which contradicts the definition of $\rho_S(x)$.
\end{proof}

>From now on and for the rest of the paper, $S$ will always be a balanced split locus:
\begin{lem}\label{rho es Lipschitz en rho lejos de lambda} 
Let $E\subset \partial\Omega $ be an open subset whose closure is compact and has a neighborhood where $\rho <\lambda_1 $.
Then $\rho_S $ is Lipschitz in $E$.
\end{lem}

\begin{proof}
The map $x\rightarrow (F(x),dF_x(r)) $ is an embedding of $V$ into $TM$. There is a constant $c$ such that for $x,y\in V$:
\begin{equation}\label{(F,dF(r)) is an embedding} 
\vert F(x)-F(y)\vert+\vert dF_x(r)-dF_y(r)\vert \geq
c\min\{\vert x-y\vert,1\}
\end{equation} 

Recall the exponential map is a local diffeomorphism before the first conjugate point.
Points $p=F((z,\rho(z))) $ for $z\in E$ have a set $R_p$ consisting of the vector $dF_{(z,\rho(z))}(r) $, and vectors coming from $V\setminus E$. 
Choose one such point $p$, and a neighborhood $U$ of $p$.
The above inequality shows that there is a constant $m$ such that:
$$
\vert dF_x(r)-dF_y(r)\vert \geq m
$$
for $x=(z,\rho(z))$ with $z\in E$ and $y =(w,\rho(w))\in Q_p$ with $w\in V\setminus E$. 
% As $H$ (thus $\varphi$) is strictly convex, for such $x$ and $y$ we also have:
% $$
% \vert \widehat{dF_x(r)}-\widehat{dF_y(r)}\vert \geq m'
% $$
By the balanced condition \ref{balanced}, any unit vector $v$ tangent to $S$ satisfies $\widehat{dF_x(r)}(v)= \widehat{dF_y(r)}(v)$ for some such $y$ and so:
$$
\widehat{dF_x(r)}(v)<1-\varepsilon
$$

%We may now apply lemma \ref{the set and the cone}.
%Define the set $A_1$ as the set of points $q\in U$ such that $Q_{q}$ contains a point $(z,t)$ for $z\in E$, and let $A=U\setminus A_1$ and $V$ be any open set whose closure is contained within $U$. 
%As $\lambda_1$ is continuous, we see $A$ is open.
%By the above, there is a cone $C$...

Thus for any vector $w$ tangent to $E$ both vectors $(w,d\rho_-(w))$ and $(w,d\rho_+(w))$ lie in a cone of fixed amplitude around the kernel of $ \widehat{dF_x(r)} $ (the hyperplane tangent to the indicatrix at $x$).
%
%Define the set $A_1$ as the set of points $q\in U$ such that $Q_{q}$ contains a point $(z,t)$ for $z\in E$.
Application of lemma \ref{the set and the cone} shows that $\rho$ is Lipschitz.

\end{proof}

\begin{lem}\label{rho es Lipschitz en rho < lambda pero cerca de un conjugado} 
Let $z_0\in \partial\Omega $ be a point such that $\rho(z_0)=\lambda_1(z_0) $.
Then there is a neighborhood $E$ of $z_0$ and a constant $C$ such that for all $z$ in $E$ with $\rho(z)<\lambda_1(z)$, $\rho$ is Lipschitz near $z$ with Lipschitz constant $C$.

%  we have 
% $$
% \vert d^{\pm}_z\rho(w)\vert<C
% $$
% for all $w\in T_z E$.
% 
% DUDA: ¿es mejor expresarlo en términos de gradientes generalizados en vez de $\vert d^{\pm}_z\rho(w)\vert $?
\end{lem}
\begin{proof}
 Let $O$ be a compact neighborhood of $(z_0,\lambda_1(z_0)) $ where special coordinates apply. 
Let $x=(z,\rho(z))\in O$ be such that $\rho(z)<\lambda_1(z) $. We can apply the previous lemma and find $\rho $ is Lipschitz near $z$. We just need to estimate the Lipschitz constant uniformly. Vectors in $R_{F(x)}$ that are of the form $dF_{y}(r)$ for $y\in V\setminus O$, are separated from $dF_x(r)$ as in the previous lemma and pose no trouble, but now there might be other vectors $dF_{y}(r)$ for $y\in O$.

Fix the metric $\langle\cdot \rangle$ in $O$ whose matrix in special coordinates is the identity.
Any tangent vector to $S$ satisfies $\widehat{dF_x(r)}(v)=\widehat{dF_y(r)}(v)$, for some $y\in O\cap Q_{F(x)} $.
A uniform Lipschitz constant for $\rho$ is found if we bound from below the angle in the metric $\langle\cdot \rangle$ between $r$ and 
$d_x F^{-1}(v)$ for any vector $v$ with this property. % $p=F(x)$.
Fix a point $y\in O$ with $F(x)=F(y)$, let $X=dF_{x}(r)$, $Y=dF_{y}(r)$ and $\alpha=\widehat{X}- \widehat{Y} $. We need to bound from below the angle between $r$ and the hyperplane $\ker\alpha $.

% Take one such point $y$, and let $X=dF_{x}(r)$, $Y=dF_{y}(r)$ and $\alpha=\widehat{X}- \widehat{Y} $.

Our goal is equivalent to proving that there is $\varepsilon_1>0 $ independent of $x$ such that:
$$
\dfrac{F_{x}^\ast \alpha (r)}{\Vert F_{x}^\ast \alpha\Vert }>\varepsilon_1
$$
which is equivalent to:
$$
\widehat{Y}(X)<1-\varepsilon_1 \Vert F_{x}^\ast \alpha\Vert
$$
in the norm $\Vert\cdot \Vert$ associated to $\langle\cdot \rangle$.

Notice first that $X$ and $Y$ belong to the indicatrix at $F(x)=F(y)$, which is strictly convex. By this and (\ref{(F,dF(r)) is an embedding}), we see that for some $\varepsilon_2>0$:
$$
\widehat{Y}(X)<1-\varepsilon_2 \Vert X-Y\Vert ^2<1-c\varepsilon_2 \Vert x-y\Vert ^2
$$
So it is sufficient to show that for some $C_1$ independent of $x$:
$$
\Vert F_x^\ast \alpha\Vert <
C_1\Vert x-y\Vert ^2
$$

Using a Taylor expansion of $\frac{\partial \varphi}{\partial x_j} $ in the second entry, we see the form $ F_x^\ast \alpha $ can be written in coordinates:
$$
\begin{array}{rcl}
F_x^\ast \alpha&=&
\left( 
\frac{\partial \varphi}{\partial x_j}(p,X)-\frac{\partial \varphi}{\partial x_j}(p,Y)
\right) 
\frac{\partial F_j}{\partial x_l}\\
&=&
\frac{\partial^2 \varphi}{\partial x_i x_j}(p,X)
\left( 
X_i - Y_i
\right) 
\frac{\partial F_j}{\partial x_l}
+O(\Vert X-Y\Vert )^2\\
&=&
\frac{\partial^2 \varphi}{\partial x_i x_j}(p,X)
\left( 
X_i - Y_i
\right) 
\frac{\partial F_j}{\partial x_l}
+O(\Vert x-y\Vert )^2
 \end{array}
$$

The argument goes as follows: we need the inequality
$\Vert F_x^\ast \alpha (v)\Vert <C_1\Vert v\Vert \Vert x-y\Vert^2$, so we want to bound the bilinear map $g_{(p,X)}$ evaluated at $X-Y$ and the vector $dF(v)$. The bound on the norm is achieved when $dF(v)$ is proportional to $X-Y$.
The map $d_x F$ is invertible, so for the vector $v=\frac{dF^{-1}(X-Y)}{\Vert dF^{-1}(X-Y) \Vert}$, we have:
$$\Vert F_x^\ast \alpha \Vert =\Vert F_x^\ast \alpha(v)\Vert $$
Thus we have:
$$
\Vert F_{(z,\rho(z))}^\ast \alpha\Vert <
C_2 \frac{\Vert X-Y\Vert^2}{\Vert dF^{-1}(X-Y) \Vert}+O(\Vert x-y\Vert)^2
<C_3 \frac{\Vert x-y\Vert^2}{\Vert dF^{-1}(X-Y) \Vert}+O(\Vert x-y\Vert)^2
$$
for constants $C_2$ and $C_3$, and it is enough to show there is $\varepsilon_3$ independent of $x$ and $y$ such that:
\begin{equation}
\Vert dF^{-1}(X-Y) \Vert>\varepsilon_3
\end{equation}
Let $G(x)= d_xF(r) $. We have:
$$
X-Y=G(x)-G(y)=dG_x(x-y)+O(\Vert x-y\Vert)
$$
so it is equivalent to show the following:
$$
\Vert dF^{-1}dG_x(x-y)\Vert >\varepsilon_4
$$
for $\varepsilon_4 $ independent of $x $ and $y$.

Assume that $(\rho(z_0),z_0)$ is conjugate of order $k$, so that $\rho(z_0)=\lambda_1(z_0)=\dots=\lambda_k(z_0)$. Thanks to Lemma \ref{landa es Lipschitz} and reducing to a smaller $O$, we can assume that $a_1=(\lambda_1(z),z)$ to $a_k=(\lambda_k(z),z)$ all lie within $O$ (some of them may coincide). Let $d_i = \lambda_i(z)-\rho(z)$ be the distance from $x$ to the $a_i$. At each of the $a_i$ there is a vector $w_i\in \ker d_{a_i}F$ such that all the $w_i$ span a $k$-dimensional subspace. Recall from section \ref{subsection: lagrangian} that we can choose $w_i$ forming an almost orthonormal subset for the above metric, in the sense that $\left\langle w_i,w_j\right\rangle=\delta_{i,j}+\varepsilon_{i,j} $.

The kernel of $d_y F$ is contained in $K=\langle \dfrac{\partial}{\partial x_{n-k+1}},\dots,\dfrac{\partial}{\partial x_{n}}\rangle $ for all $y\in O$, and thus $K=\langle w_1,\dots,w_k\rangle$.
Write $w_i=\sum_{j\geq n-k+1} w_i^j \dfrac{\partial }{\partial x^j}$.
Then we have $\dfrac{\partial }{\partial x^1}\dfrac{\partial }{\partial w_i}F(a)=z_i+R_i(a)$, for $z_i=\sum w_i^k \dfrac{\partial }{\partial y^k}$, $\|R_i(a)\|<\varepsilon$ and $a\in O$.
We deduce $\dfrac{\partial }{\partial w_i}F(x)=\dfrac{\partial }{\partial w_i}F(a_i)+d_i (z_i+v_i)=d_i (z_i+v_i)$ for $\|v_i\|<\varepsilon$.

By the form of the special coordinates, $x-y\in K$. Let $x-y=\sum b_iw_i$. 
Since $|w_i|$ is almost $1$, there is an index $i_0$ such that $|b_{i_0}|>\frac{1}{2n}\|x-y\|$. 
% We claim there is a constant $M$ such that $d_i<M \Vert x-y\Vert$ for at least one $i$.
We have the identity:
$$
0=F(y)-F(x)=d_x F(y-x)+O(\|x-y\|^2)
= \sum b_i d_i (z_i+ v_i ) + O(\|x-y\|^2)
$$
% for all $i$ and some constant $C$ depending on bounds for partial derivatives of $F$.
Multiplying the above by $\pm z_j$,
% >From the equality $(F(x)-F(y))\cdot z_j=0$, 
we deduce $d_j| b_j |= -\sum |b_i| d_j(\varepsilon_{i,j}+v_i z_j) + O(\|x-y\|^2)$, which leads to 
\begin{equation}\label{desigualdad cerca de un punto conjugado} 
 \sum |b_i| d_i<C_4\|x-y\|^2
\end{equation} 
% For this index we have $\frac{1}{2n}\|x-y\|d_{i_0}< C\|y-x\|^2$, and thus $d_{i_0}< C_2 |b_{i_0}|$.

At the point $x$, the image by $d_{x}F $ of the unit ball $B_x V$ in $T_x V$ is contained in a neighborhood of $Im (d_{a_{i}}F)$ of radius $2d_{i}$. 
We use the identity
$$
\Vert dF^{-1}dG_x(\frac{ x-y}{\Vert x-y\Vert} )\Vert^{-1} =
\sup\{ t: tdG_x(\frac{ x-y}{\Vert x-y\Vert} ) \in d_x F(B_x V) \}
$$
We can assume the distance between the vectors $dG_{x}(\frac{x-y}{\Vert x-y\Vert})$ and $\sum \frac{ b_i}{\Vert x-y\Vert}z_i$ is smaller than $\frac{1}{4n} $. 
% (DONE: $dG_{x}(\frac{ y-x}{\vert y-x\vert})=\sum \frac{ b_i}{\vert y-x\vert}z_i$+ruido)
In particular, looking at the $i_0$ coordinate chosen above,
we see that the vector $dG_{x}(\frac{ x-y}{\Vert x-y\Vert})$ needs to be rescaled at least by the amount $8nd_{i} $ in order to fit within the image of the unit ball. 

% We deduce
% At the point $x$, the image by $d_{x}F $ of the unit ball at $T_x V$ is contained in a neighborhood of $Im (d_{a_{i_0}}F)$ of radius $d_{i_0}$, while $dG_{x}(w) $ is at most a distance $C(|x-x_0| +|w-w_i|)$ from $dG_{x_{0}}(w_i) $ for any $w$ close to $w_i$. 
% Thus we can assume the vector $dG_{x}(\frac{ x-y}{\vert x-y\vert})$ is at most a distance $\frac{1}{4n} $ from the unit vector $\sum \frac{ b_i}{\vert x-y\vert}z_i$ and needs to be rescaled by the amount $c_2 d_{i_0} $ in order to fit within the image of the unit ball. We deduce
% any vector $Z$ sufficiently close to $d_xG(v) $ is a distance satisfies:
% $$
% \Vert dF^{-1}(Z)\Vert >\frac{ c}{\Vert x_n-y_n\Vert}
% $$
% so
$$
\Vert dF^{-1}dG_x(\frac{ x-y}{\Vert x-y\Vert} )\Vert 
>\frac{1}{8n d_{i_0}}
>\frac{ |b_{i_0}|}{8nC_4\Vert x-y\Vert^2 }
>\frac{\varepsilon_4}{\Vert x-y\Vert }
$$
for $\varepsilon_4=\frac{1}{16n^2C_4}>0$, which is the desired inequality.
\end{proof}

\begin{proof}[Proof of Lemma \ref{rho is lipschitz}]
We prove that $\rho$ is Lipschitz close to a point $z^0$. Let $E$ be a neighborhood of $z^0$ such that $\lambda_1 $ has Lipschitz constant $L$, and $\rho $ has Lipschitz constant $K $ 
% \begin{equation}\label{D rho menor que K} 
%  \vert d^{\pm}_z\rho(w)\vert<K
% \end{equation} 
for all $z\in E$  such that $\rho(z)<\lambda(z)$. % and $w\in T_z E$.
% fix coordinates in $U$ such that $E\cap \partial \Omega $ is a hyperplane 
Let $z^1,z^2\in E$ be such that $\rho(z^1)<\rho(z^2)$.

If $\rho (z^1)=\lambda_1(z^1)$ we can compute
$$
\vert \rho (z^2)- \rho (z^1)\vert=\rho (z^2)- \rho (z^1)<\lambda(z^2)-\lambda(z^1)<L\vert z^2-z^1\vert
$$
where $L$ is a Lipschitz constant $L$ for $\lambda $ in $U$.

Otherwise take a linear path with unit speed $\xi:[0,t]\rightarrow \partial \Omega$ from $z^1$ to $z^2$ and let $a$ be the supremum of all $s$ such that $\rho(\xi(s))<\lambda(\xi(s)) $. Then
$$
\vert \rho (z^2)- \rho (z^1)\vert<
\vert \rho (z^2)- \rho (\xi(a))\vert+\vert \rho (\xi(a))- \rho (z^1)\vert
$$

\noindent The second term can be bound:
% by \ref{rho es Lipschitz en rho < lambda pero cerca de un conjugado}, with the constant $K$ there:
% using (\ref{D rho menor que K}):
$$
\vert \rho (\xi(a))- \rho (z^1)\vert<K a
$$

\noindent If $\rho (z^2)\geq \rho (\xi(a)) $, we can bound the first term as
$$
\vert \rho (z^2)- \rho (\xi(a))\vert=\rho (z^2)- \rho (\xi(a))<
\lambda(z_2)-\lambda(\xi(a))<L\vert t-a\vert
$$
while if  $\rho (z^2)< \rho (\xi(a)) $, we have
$$
\vert \rho (z^2)- \rho (z^1)\vert<\vert \rho (\xi(a))- \rho (z^1)\vert
$$
so in all cases, the following holds:
$$
\vert \rho (z^2)- \rho (z^1)\vert<\max\{L,K\}t<\max\{L,K\}\vert z^2 - z^1\vert
$$
 
\end{proof}

\section{Proof of the main theorems.}\label{section:proof of the main theorems}
% \section{The current $T$}
Take the function $h$ associated to $S$ as in definition \ref{u associated to S}. At a cleave point $x$ there are two geodesics arriving from $\partial\Omega$; each one yields a value of $h$ by evaluation of $\tilde{u}$.
The \emph{balanced} condition implies that $\widehat{X}_1(v)=\widehat{X}_2(v)$ for the speed vectors $X_1$ and $X_2$ of the characteristics reaching $x$ and any vector $v$ tangent to $S$.
Furthermore, $\widehat{X}$ is exactly $dh$, so the difference of the values of $h$ from either side is constant in every connected component of the cleave locus.
% We would like to find out if is zero or not.

We define an $(n-1)$-current $T$ in this way: Fix an orientation $\mathcal{O} $ in $\Omega$. For every smooth $(n-1)$ differential form $\phi$, restrict it to the set of cleave points $\mathcal{C}$ (including degenerate cleave points). In every component $\mathcal{C}_{j}$ of $\mathcal{C}$ compute the following integrals

\begin{equation}\label{}
\int_{\mathcal{C}_{j,i}}h_{i}\phi \qquad i=1,2
\end{equation} 
where $\mathcal{C}_{j,i}$ is the component $\mathcal{C}_{j}$ with the orientation induced by $\mathcal{O} $ and the incoming vector $V_{i}$, and $h_{i}$ for $i=1,2$ are 
the limit values of $h$ from each side of $\mathcal{C}_{j}$.
% the local distance functions on the two sides of $\mathcal{C}_{j}$.

% and $e_{2}$ are the $n-1$ vectors obtained by performing the inner product of $\overrightarrow{k}$ with the dual one-form to the two vector fields $V_{1}$ and $V_{2}$ defined at $\mathcal{C}_{j}$, and $h_{i}$ are the values of $u$ computed from either side. 
% We define the current $T(\phi)$ to be the sum:
% 
% \begin{equation}\label{definition of T}
% T(\phi)=\sum_{j}\int_{\mathcal{C}_{j}}\phi(e_{1})h_{1}+\int_{\mathcal{C}_{j}}\phi(e_{2})h_{2}
% \end{equation} 

We define the current $T(\phi)$ to be the sum:

\begin{equation}\label{definition of T}
T(\phi)=\sum_{j}\int_{\mathcal{C}_{j,1}}h_{1}\phi
+\int_{\mathcal{C}_{j,2}}h_{2}\phi
=\sum_{j}\int_{\mathcal{C}_{j,1}}(h_{1}-h_{2})\phi
\end{equation} 

The function $h$ is bounded and the $\Hnuno$ measure of $\mathcal{C} $ is finite (thanks to lemma \ref{rho is lipschitz}) so that $T$ is a real flat current that represents integrals of test functions against the difference between the values of $h$ from both sides.
% It's not hard to see that the integrand $h_{1}-h_{2} $ is constant in each $\mathcal{C}_{j} $.

If $T=0$, we can apply lemma \ref{una unica solucion por caracteristicas continua en un conjunto grande} and find $u=h$.
% If, in some situation, we have $T=0$, then $\rho_S$ gives the same value as $\rho_{Sing}$ when evaluated at the points $x$ such that $F(\rho_S(x),x)$ and $F(\rho_{Sing}(x),x)$ are cleave points. This set is dense in $\partial\Omega$, so $\rho_S=\rho_{Sing}$, $S=Sing$ and $u=h$.

We will prove later that the boundary of $T$ as a current is zero.
Assume for the moment that $\partial T=0 $. It defines an element of the homology space $H_{n-1}(\Omega)$ of dimension $n-1$ with real coefficients. We can study this space using the long exact sequence of homology with real coefficients for the pair $ (\Omega,\partial\Omega)$:
\begin{eqnarray}\label{long exact sequence} 
0\rightarrow H_{n}(\Omega)\rightarrow H_{n}(\Omega,\partial\Omega)\rightarrow\notag\\
H_{n-1}(\partial\Omega)\rightarrow H_{n-1}(\Omega)\rightarrow H_{n-1}(\Omega,\partial\Omega)\rightarrow \dots
\end{eqnarray}

\subsection{Proof of Theorem \ref{maintheorem0}.}
We prove that under the hypothesis of \ref{maintheorem0}, the space $H_{n-1}(\Omega)$ is zero, and then we deduce that $T=0$.

As $\Omega$ is open, $H_{n}(\Omega)\approx 0$. As $\Omega$ is simply connected, it is orientable, so we can apply Lefschetz duality with real coefficients (\cite[3.43]{Hatcher}) which implies:
$$H_{n}(\Omega,\partial\Omega)\approx H^{0}(\Omega)$$
and  
$$H_{n-1}(\Omega,\partial\Omega)\approx H^{1}(\Omega)=0$$

As $\partial\Omega$ is connected, we deduce $H_{n-1}(\Omega)$ has rank $0$, and $T=\partial P$ for some $n$-dimensional flat current $P$. The flat top-dimensional current $P$ can be represented by a density $f\in L^{n}(\Omega)$ (see \cite[p 376, 4.1.18]{Federer}):

\begin{equation}
P(\omega)=\int _{\Omega}f\omega,,\qquad \omega\in \Lambda^{n}(\Omega)
\end{equation} 

We deduce from (\ref{definition of T}) that the restriction of $P$ to any open set disjoint with $S$ is closed, so $f$ is a constant in such open set. As $\Omega\setminus S$ is open and connected, $f$ is constant $a.e.$, and $T=0$.

\subsection{Proof of Theorem \ref{maintheorem1}.}
Assume now that $\partial\Omega$ has $k$ connected components $\Gamma_{i}$. We look at (\ref{long exact sequence}), and recall the map $H_{n-1}(\partial\Omega)\rightarrow H_{n-1}(\Omega)$ is induced by inclusion.
We know by Poincar\'e duality that $H_{n-1}(\partial\Omega)$ is isomorphic to the linear combinations of the fundamental classes of the connected components of $\partial\Omega$ with real coefficients.
We deduce that $H_{n-1}(\Omega)$ is generated by 
% the Poincar\'e duals of the elements of $H_{0}(\partial\Omega)$, which are 
the fundamental classes of the connected components of $\partial\Omega$, and that it is isomorphic to the quotient of all linear combinations by the subspace of those linear combinations with equal coefficients.
Let
$$
R=\sum a_{i}\left[ \Gamma_{i}\right] 
$$
be the cycle to which $T$ is homologous (the orientation of $\Gamma_{i} $ is such that, together with the inwards pointing vector, yields the ambient orientation).

If we define $a(x)=a_{i}$, $\forall x\in \Gamma_{i}$, solve the HJ equations with boundary data $g-a$ and compute the corresponding current $\widehat{T} $, we see that $\widehat{T} = T - j_{\sharp}R$, where $j$ is the retraction $j$ of $\Omega$ onto $S$ that fixes points of $S$ and follows characteristics otherwise. Then the homology class of $\widehat{T}$ is zero, and we can prove $\widehat{T}=0 $ as before. It follows that $S$ is the singular set to the solution of the Hamilton-Jacobi equations with boundary data $g-a $.

\subsection{Proof of Theorem \ref{maintheorem2}.}
For this result we cannot simply use the sequence (\ref{long exact sequence}). We first give a procedure for obtaining balanced split loci in $\Omega$ other than the cut locus. 

A function $a:[\partial \widetilde{\Omega}]\rightarrow \RR $ that assigns a real number to each connected component of $\partial \widetilde{\Omega} $ is \emph{equivariant} iff for any automorphism of the cover $\varphi$ there is a real number $c(\varphi)$ such that $a\circ \varphi =a+c(\varphi)$.

A function $a:[\partial \widetilde{\Omega}]\rightarrow \RR $ is \emph{compatible} iff $\widetilde{g}-a$ satisfies the compatibility condition (\ref{compatibility condition}).

An equivariant function $a$ yields a group homomorphism from $\pi_1(\Omega, \partial\Omega)$ into $\RR$ in this way:
\begin{equation}\label{from a function a to an element of Hom(pi,R)} 
\sigma\rightarrow a(\widetilde{\sigma}(1))-a(\widetilde{\sigma}(0)) 
\end{equation} 
where $\sigma:[0,1]\rightarrow \Omega $ is a path with endpoints in $\partial\Omega $  and  $\widetilde{\sigma}$ is any lift to $\widetilde{\Omega} $ . The result is independent of the lift because $a$ is equivariant. 
On the other hand, choosing an arbitrary component $[\Gamma_0]$ of $\partial\Omega$ and a constant $a_0=a([\Gamma])$, the formula:
\begin{equation}\label{from an element of Hom(pi,R) to a function} 
[\Gamma]\rightarrow a([\Gamma_0])+l(\pi\circ\tilde{\sigma}),\text{ for any path }\tilde{\sigma}\text{ with }\tilde{\sigma}(0)\in \Gamma_0,\sigma(1)\in \Gamma
\end{equation} 
assigns an equivariant function $a$ to an element $l$ of $Hom(\pi_1(\Omega, \partial\Omega), \RR)\sim H^{1}(\Omega, \partial\Omega)  $.

Up to addition of a global constant, these two maps are inverse of one another, so there is a one-to-one correspondence between elements of 
$H^{1}(\Omega, \partial\Omega) $
and equivariant functions $a$ (with $a+c$ identified with $a$ for any constant $c$).
The compatible equivariant functions up to addition of a global constant can be identified with an open subset of $H^{1}(\Omega, \partial\Omega) $ that contains the zero cohomology class.

% Clearly, the map $c$ is an homomorphism from $\pi_1(\Omega)$ into $\RR$, and thus a member of $H^1(\Omega)$. 
% 
% de algun modo, $a$ deberia dar un homomorfismo de $\pi_1(\Omega, \partial\Omega)$ hasta $\RR$.

Let $\widetilde{\Omega} $ be the universal cover of $\Omega $.
We can lift the Hamiltonian $H$ to a function $\widetilde{H} $ defined on $T^\ast\widetilde{\Omega}$ and the function $g$ to a function $\tilde{g} $ defined on $\partial \widetilde{\Omega}$.
The preimage of a balanced split locus for $\Omega $, $H$ and $g$ is a balanced split locus for $\widetilde{\Omega} $, $\widetilde{H} $ and $\tilde{g} $ that is invariant by the automorphism group of the cover, and conversely, a balanced split locus $\widetilde{S}$ in $\widetilde{\Omega} $ that is invariant by the automorphism group of the cover descends to a balanced split locus on $\Omega $.
% We remark that even though $\widetilde{\Omega} $ may not be compact, the structure theorem \ref{structure up to codimension 3} still holds, as $\widetilde{S}$ is locally the preimage of $S$ under a diffeomorphism.

% Recall we are assuming $\widetilde{\Omega} $ is compact, so the previous theorem shows that all balanced split loci in $\widetilde{\Omega} $ appear as the singular locus of the solution of a HJ boundary problem with data $\widetilde{h}=\widetilde{g}+a $, where the function $a$ is constant in every connected component of $\partial \widetilde{\Omega}$.
% We are interested in those balanced split loci that descend to balanced split loci on $\Omega $.
% , and would like to remove the compactness hypothesis on $\widetilde{\Omega} $.

% We notice that a balanced split locus in $\Omega $ lifts to a balanced split locus in $\widetilde{\Omega} $ that is invariant by the automorphism group of the cover, and a
% balanced split locus $\widetilde{S}$ in $\widetilde{\Omega} $ that is invariant by the automorphism group of the cover descends to a balanced split loci on $\Omega $.

Any function $a$ that is both equivariant and compatible can be used to solve the Hamilton-Jacobi problem $\widetilde{H}(p,du(p))=1$ in $\widetilde{\Omega} $ and $u(p)=\widetilde{g}(p)-a(p)$. 
If $\pi_1(\Omega)$ is not finite, $\widetilde{\Omega}$ will not be compact, but this is not a problem (see remark 5.5 in page 125 of \cite{Lions}).
The singular set is a balanced split locus that is invariant under the action of $\pi_1(\Omega) $ and hence it yields a balanced split locus in $\Omega $. 
We write $S[a]$ for this set.
It is not hard to see that the map $a\rightarrow S[a]$ is injective.

Conversely, a balanced split locus in $\Omega $ lifts to a balanced split locus $\widetilde{S}$ in $\widetilde{\Omega} $. 
% Also the sets $R_p$, the characteristics, and the current $T$ are lifted to $\widetilde{\Omega} $.
The reader may check that the current $T_{\widetilde{S}}$ is the lift of $T_S$, and in particular it is closed. 
As in the proof of Theorem \ref{maintheorem1}, we have $H^1(\widetilde{\Omega})=0$, and we deduce
$$
T_{\widetilde{S}} = \sum_j a_j[\Lambda_j] +\partial P
$$
where $\Lambda_j$ are the connected components of $\partial \widetilde{\Omega}$. 
% Each component $\Gamma_i$ may have one or several $\Lambda_i$ as preimages...

% Indeed, $a$ is chosen by taking a representative of the homology class that $T_{\widetilde{S}}$ represents in the space $H_{n-1}(\widetilde{\Omega})$.
% the values that $a$ takes in a connected component is the coefficient in that component of a representative for the homology class that $\widetilde{T}$ represents. 
This class is the lift of the class of $T\in H_{n-1}(\Omega)$ and thus it is invariant under the action of the group of automorphisms of the cover. 
Equivalently, the map defined in (\ref{from a function a to an element of Hom(pi,R)}) is a homomorphism.
% Thus the differences $a(\Gamma_0)-a(\Gamma_1)$ for different connected components of $\partial\widetilde{\Omega} $ that lie over the same connected components of $\partial\Omega $ vanish and thus $a$ is equivariant.
Thus $a$ is equivariant.
Similar arguments as before show that $S=S[a]$.

Thus the map $a\rightarrow S[a] $ is also surjective, which completes the proof that there is a bijection between equivariant compatible functions $a:[\partial \widetilde{\Omega}]\rightarrow \RR $ and balanced split loci.

\section{Proof that $\partial T=0$}
It is enough to show that $\partial T=0$ at all points of $\Omega$ except for a set of zero $(n-2)$-dimensional Hausdorff measure. This is clear for points not in $S$.
Due to the structure result \ref{structure up to codimension 3}, we need to show the same at cleave points (including degenerate ones), edge points and crossing points.
Along the proof, we will learn more about the structure of $S$ near those kinds of points.

Throughout this section, we assume $n=dim(\Omega)>2$. This is only to simplify notation, but the case $n=2$ is covered too. We shall comment on the necessary changes to cover the case $n=2$, but do not bother with the simple case $n=1$.

\subsection{Conjugate points of order $1$.}

We now take a closer look at points of $A(S)$ that are also conjugate points of order $1$. Fortunately, because of \ref{structure up to codimension 3} we do not need to deal with higher order conjugate points.
In a neighborhood $O$ of a point $x^{0}$ of order $1$, in the special coordinates of section \ref{subsection: special coordinates}, we have $x^0=0$ and $F$ looks like:
\begin{eqnarray}\label{local form for order 1}
F(x_{1},x_{2},\dots,x_{n}) =
(x_{1},x_{2},\dots,F_n(x_{1},\dots,x_{n}))
\end{eqnarray}

Let $\widetilde{S}$ be the boundary of $A(S)$, but without the points $(0,z)$ for $z\in\partial\Omega $.
It follows from \ref{rho is lipschitz} that $\widetilde{S}$ is a Lipschitz graph on coordinates given by the vector field $r$ and $n-1$ transversal coordinates. 
% It follows, from example from \ref{the set and the cone}, 
It is not hard to see
that it is also a Lipschitz graph $x_{1}=\tilde{t}(x_{2},\dots,x_{n})$ in the above coordinates $x_{i}$, possibly after restricting to a smaller open set.

Because of  Lemma \ref{rho es menor que landa1}, we know $x^0$ is a first conjugate point, so we can assume that $O$ is a coordinate cube $\prod (-\varepsilon_i,\varepsilon_i)$, and that $F$ is a diffeomorphism when restricted to $\{x_1=s\} $ for $s<-\varepsilon_1/2$.
% Recall from \ref{F near order k in special coords} that $\frac{\partial h}{\partial x_n}=0$ and $\frac{\partial^2 h}{\partial x_1 x_n}\neq 0$. As $x^0$ is a first conjugate point, we have $\frac{\partial^2 h}{\partial x_1 x_n}< 0$.

\begin{dfn}\label{univocal} 
 A set $O\subset V$ is \emph{univocal} iff for any $ p\in \Omega$ and $x^{1},x^{2}\in Q_{p}\cap O$ we have $\tilde{u}(x^{1})=\tilde{u}(x^{2})$.
\end{dfn}

\paragraph{Remark.} The most common case of univocal set is a set $O$ such that $F\vert O$ is injective.

\begin{lem}\label{uniqueness near order 1 points} 
 Let $x^{0}\in V$ be a conjugate point of order 1. Then $x^{0}$ has an univocal neighborhood.
%  such that $F(O)$ is a neighborhood of $F(x^0)$.
%%% TODO: do I need F(O) open???
% , and $p=F(x)$ has a neighborhood $U$ such that $\forall q\in U$, $x^{1},x^{2}\in Q_{q}\cap O$ implies $\tilde{u}(x^{1})=\tilde{u}(x^{2})$.
\end{lem}
\begin{proof}
Let $O_1$ and $U_1$ be neighborhoods of $x^0$ and $F(x^0)$ where the special coordinates (\ref{local form for order 1}) hold;
let $x_{i}$ be the coordinates in $O_1$ and $y_{i}$ be those in $U_1$.

% Thanks to the inequality $\frac{\partial^2 h}{\partial x_1 x_n}< 0$, we know that the function
% \begin{equation}\label{a function with negative second derivative} 
% t\rightarrow \widehat{dF_{l(t)}(r)}(\frac{\partial }{\partial y_{1}}) 
% \end{equation} 
% where $l(t)=(b_1,\dots,b_{n-1},t)$ has negative second derivative for all $(b_1,\dots,b_{n-1},t) \in U_1$. On the other hand, in the coordinates used (section \ref{subsection: special coordinates}), all derivatives $\frac{\partial ^k }{\partial x_{1}^k} $ of the function $x\rightarrow \widehat{dF_x(r)} \frac{\partial }{\partial y_{1}}$ are zero at $p$.
Choose smaller $U\subset U_1$ and $O\subset F^{-1}(U)\cap O_1$ so that
we can assume that if a point $x'\in V\setminus O_1$ maps to a point in $U$, then for the vector $Z=dF_{x'}(r)$ we have
\begin{equation}\label{separation between O and V minus O_1} 
\hat{Z}(\frac{\partial }{\partial y_{1}}) <
\hat{X}(\frac{\partial }{\partial y_{1}})
\end{equation} 
for any $X=dF_{x}(r)$ with $x\in O$ and also
\begin{equation}\label{vectors from O_1 are mostly vertical} 
\hat{Y}(\frac{\partial }{\partial y_{1}}) > 1-k
\end{equation} 
for some $k>0$ sufficiently small and all $Y=dF_{y'}(r)$ for $y'\in O_1$.

Take $x^{1},x^{2}\in Q_{q}\cap O$ for $q\in U$.
The hypothesis $x^{1},x^{2}\in Q_{q}$ implies $q=F(x^{1})=F(x^{2})$, and so $x^{1}_{j}=x^{2}_{j} $ follows for all $j<n$. Let us write $a_{j}=x^{1}_{j}=x^{2}_{j} $ for  $j<n$, $s^{1}=x^{1}_{n}$ and $s^{2}=x^{2}_{n}$.
Fix $a_{2},\dots,a_{n-1}$ and consider the set
$$
H_{a}=
\left\lbrace 
x\in O: x_{i}=a_{i};i=2,\dots, n-1
\right\rbrace
$$
% parametrized by $(s,t)\rightarrow (s,a_{2},\dots,a_{n-1},t)$. 
Its image by $F$ is a subset of a plane in the $y_i$ coordinates: 
$$
L_{a}=\left\lbrace y\in U: y_{i}=a_{i}, i=2,\dots, n-1\right\rbrace
$$
Points of $O_1$ not in $H_{a}$ map to other planes.
If $n=2$, we keep the same notation, but the meaning is that $H_a=O$ and $L_a=V$.

There is $\varepsilon>0$ such that for $t<-\varepsilon/2$, the line $\{x_1=t\}\cap H_a $ maps diffeomorphically to $\{y_1=t\}\cap L_a$.
% We can take $O\cap \{x_1>-\varepsilon\} $ instead of $O$.

Due to the comments at the beginning of this section, $\widetilde{S}$ is given as a Lipschitz graph $x_{1}=\tilde{t}(x_{2},\dots,x_{n})$.
The identity $a_{1}=\tilde{t}(a_{2},\dots,a_{n-1},s^{i})$ holds for $i=1,2$ because $x^{1},x^{2}\in Q_{q}$. We define a curve $\sigma:[s^{1},s^{2}]\rightarrow \widetilde{S}$ by $\sigma(s)=(\tilde{t}(a_{2},\dots,a_{n-1},s),a_{2},\dots,a_{n-1},s) $. The image of $\sigma$ by $F$ stays in $S$, describing a closed loop based at $q$; we will establish the lemma by examining the variation of $\tilde{u}$ along $\sigma$.

For $i=1,2$, let $\eta^i:(-\varepsilon_i,a_1]\rightarrow H_a$ given by $\eta^i(t)=(t,a_2,\dots,a_{n-1},s^i)$ be the segments parallel to the $x_{1}$ direction that end at $x^{i}$, defined from the first point in the segment that is in $O$.
We can assume that the intersection of $O$ with any line parallel to $\frac{\partial}{\partial x_1}$ is connected, and that the intersection of $U$ with any line parallel to $\frac{\partial}{\partial y_1}$ is connected too.
We can also assume $\varepsilon_i<\varepsilon$.

Let $D$ be the closed subset of $H_{a}$ delimited by the Lipschitz curves $\eta^1$, $\eta^2$ and $\sigma$, and let $E$ be the closed subset of $L_{a}$ delimited by the image of $\eta^1$ and $\eta^2$. 

We claim $D$ is mapped onto $E$.
First, no point in $int(D) $ can map to the image of the two lines, cause this contradicts either $\rho\leq \lambda_{1}$, or the fact that $\rho(a_{2},\dots,a_{n-1},s^{i})$ is the first time that the line parallel to the $x_{1}$ direction hits $\widetilde{S}$, for either $i=1$ or $i=2$.
We deduce $D$ is mapped into $E$.
% Thus $\partial D$ maps to $\partial E$ and $int(D) $ to $int(E)$. 
% Second, $F(D)$ contains a neighborhood of $c^i\vert_{(-\varepsilon_1,a_1)}$ for $i=1,2$.

Now assume $G=E\setminus F(D)$ is nonempty, and contains a point $p=(p_1,\dots,p_n)$. 
% There are three possibilities for $p$, and our next task is ruling them out:
% \begin{itemize}
% \item 
If $Q_p$ contains a point $x\in O_1\setminus F(D)$, following the curve $t\rightarrow (t,x_2,\dots,x_n)$ backwards from $x=(x_1,x_2,\dots,x_n)$, we must hit either a point in the image of $\eta^i\vert_{(-\varepsilon_1,a_1)}$ (which is a contradiction with the fact that both $(t,\dots,x_n) $ for $t<x_1$ and $(t,a_2,\dots,a_{n-1},s^i) $ for $t<a_1$ are in $A(S)$; see definition \ref{splits}), or the point $q$ (which contradicts  (\ref{vectors from O_1 are mostly vertical})). Thus for any point $p\in G$, we have $Q_p\subset V\setminus O_1$.  

% \item 
Now take a point $p\in \partial G$, and
pick up a sequence approaching it from within $G$ and contained in a line with speed vector $\frac{\partial }{\partial y_{1}} $. By the above, the set $Q$ for points in this sequence is contained in $V\setminus O_1$. We can take a subsequence carrying a convergent sequence of vectors, 
 and thus $R_{p}$ has a vector of the form $dF_{x^{\ast}}(r) $ for $x^{\ast}\in D\subset O$.
This violates the balanced condition, because of (\ref{separation between O and V minus O_1}).
This implies $\partial G=\emptyset$, thus $G=\emptyset$ because $E$ is connected and $F(D)\neq\emptyset $.

% Let $v$ be the tangent vector in $L_a$ such that $\widehat{dF(c^1)}(v)=\widehat{dF(c^2)}(v)$. 
% $v$ is close to $\frac{\partial }{\partial y_{1}}$, $s$ is Lipschitz and $U$ is small enough, so we can assume the function
% \begin{equation}\label{another function with negative second derivative}
% t\rightarrow \widehat{dF_{\gamma(t)}(r)}(v) 
% \end{equation}
% has negative second derivative for all $t$. It has the same value at both $t^1$ and $t^2$ and thus it has lower value for any $t$ outside the interval $[t^1,t^2]$ than it has for $t$ inside the interval.

% Pick up a sequence approaching a point of $\partial G$ different from $q$ from within $G$ and contained in a line with speed vector $v $. The set $R$ for points in this sequence has vectors not from $O$. We can take a subsequence carrying a convergent sequence of vectors from $V\setminus O$, while the point hit
% $q^{\ast}$ is in $F(D)\setminus \{q\}$, and thus $R_{q^{\ast}}$ has a vector of the form $dF_{x^{\ast}}(r) $ for $x^{\ast}=(s^{\ast},a_2,\dots,a_{n-1},t^{\ast})\in S$ with $t^1<t^{\ast}<t^2$.
% This violates the balanced condition, whether the limiting vector is in $V\setminus O_1$ (because of \ref{separation between O and V minus O}) or in $O_1\setminus F(D)$ 
% (because $F$ maps sets $H_a$ to sets $L_a$ and \ref{another function with negative second derivative} has lower value for $t$ outside $[t^1,t^2]$ than it has for $t^{\ast}$).

% \item 
Finally, we claim there are no vectors coming from $V\setminus O_1$ in $R_p$ for $p\in int(E)$. 
The argument is as above, but we now approach a point with a vector from $V\setminus O_1$ in $R_p$ within $E=F(D)$ and with speed $-\frac{\partial }{\partial y_{1}}$. The approaching sequence may be chosen so that it carries a convergent sequence of vectors from $F(D)$, and again (\ref{separation between O and V minus O_1}) gives a contradiction with the balanced condition.

% \end{itemize}

% Otherwise, there is a sequence carrying vectors from $O$ with incoming speed $-\frac{\partial }{\partial y_{1}} $ converging to a point that has a vector from $V\setminus O$, which by \ref{separation between O and V minus O} contradicts the balanced condition.

% We notice that $F\vert_{H_{a}}:H_{a}\rightarrow L_{a}$ has the essential properties of an exponential map (no es verdad, $r$ no esta en $TH_{a}$)
% 
% ...

% The identity $a_{1}=s(a_{2},\dots,a_{n-1},t^{i})$ holds for $i=1,2$. We define a curve $\gamma:[t^{1},t^{2}]\rightarrow S$ by $\gamma(t)=(s(a_{2},\dots,a_{n-1},t),a_{2},\dots,a_{n-1},t) $, and compute:
We now compute:
\begin{equation}\label{u dere- u zurda es una integral} 
 \tilde{u}(x^{1})-\tilde{u}(x^{2})
=\int_{s_{1}}^{s_{2}} \frac{d(\tilde{u}\circ \sigma)}{ds}
% =\int_{t_{1}}^{t_{2}} d\tilde{u}(\gamma')
=\int_{\sigma} d\tilde{u}
\end{equation} 

% We claim the domain of $\gamma$ can be expressed as a \emph{disjoint union}:
% $$
% [t^{1},t^{2}]=\bigcup_{i\in I} A_{i}\cup \bigcup_{i\in I} B_{i} \cup N
% $$
% where 
% $$
% \int_{A_{i}} d\tilde{u}(\gamma') = -\int_{B_{i}} d\tilde{u}(\gamma')
% $$
% and
% $$
% \int_{N} \vert d\tilde{u}(\gamma')\vert = 0
% $$
% and the lemma follows.

The curve $F\circ\sigma$ runs through points of $S$. If $F(\sigma(s))$ is a cleave point, then $F\circ\sigma$ is a smooth curve near $s$. We show that cleave points are the only contributors to the above integral. 
If a point is not cleave, either it is the image of a conjugate vector, or has more than $2$ incoming geodesics. As $F\circ\sigma$ maps into $int(E)$, all vectors in $R_{F(\sigma(s))}$ come from $O$.

Let $N$ be the set of $s$ such that $\sigma(s)$ is conjugate. We notice that $\sigma(s)$ is not an A2 point for $s\in N $. This is proposition 6.2 in \cite{Nosotros}, and is a standard result for cut loci in Riemannian manifolds. This means that at those points the kernel of $dF$ is contained in the tangent to $\widetilde{S}$. The intersection of $\widetilde{S}$ with the plane $H_a$ is the image of the curve $\sigma $. Thus, for $s\in N $ the tangent to the curve $\lambda_1$ is the kernel of $d_{\sigma(s)}F$.
If $\sigma$ is differentiable at a point $s$ we deduce, thanks to \ref{rho es menor que landa1}, that the tangent to the curve $\lambda_1$ is the kernel of $d_{\sigma(s)}F$.

We now use a \emph{variation of length} argument to get a variant of the Finsler Gauss lemma. Let $c=(l,w) $ be a tangent vector to $V\subset \RR\times\partial\Omega $ at the point $x=(t,z)$, and assume $d_xF(c)=0$.
We show that this implies $d\tilde{u}(c)=0$.
Let $\gamma_s$ be a variation through geodesics with initial point in $z(s)\in \partial\Omega $ and the characteristic vector field at $z(s)$ as the initial speed vector, such that $\frac{\partial}{\partial s}z(s)=w$, and with total length $t+sl$ .
By the first variation formula and the equation for the characteristic vector field at $\partial\Omega$, the variation of the length of the curve $\gamma_s$ is $\frac{\partial     \varphi}{\partial v}(p,d_x F(r_x)) \cdot d_x F(c) - \frac{\partial \varphi}{\partial v}(p,d_z F(r)) \cdot w = -dg(w) $, and by the definition of $\gamma_s$, it is also $l$. 
We deduce $l= -dg(w)$, and thus $d\tilde{u}(c)=l+dg(w)=0$. 

It follows that $d_{\sigma(s)}F(\sigma'(s))=0$ at points $s\in N$ where $\sigma$ is differentiable.
As $\sigma$ is Lipschitz, the set of $s$ where it is not differentiable has measure $0$, and we deduce:
$$
\int_{N}  d\tilde{u}(\sigma') = 0
$$

$N$ is contained in the set of points where $d(F\circ \sigma)$ vanishes. Thus, by the Sard-Federer theorem, the image of $N$ has Hausdorff dimension $0$.

Let $\Sigma_{2}$ be the set of points in $L_{a}$ with more than $2$ incoming geodesics. From the proof of \cite[7.3]{Nosotros}, we see that the tangent to $\Sigma_{2}$ has dimension $0$ and thus $\Sigma_{2}$ has Hausdorff dimension $0$.

As $F$ is non-singular at points in $[s_1,s_2]\setminus N$, the set of $s$ in $[s_1,s_2]\setminus N$ mapping to a point in $\Sigma_{2}\cup N$ has measure zero.

Altogether, we see that the integral (\ref{u dere- u zurda es una integral}) can be restricted to the set $C$ of $s$ mapping to a cleave point. $C$ is an open set and thus can be expressed as the disjoint union of a countable amount of intervals. Let $A_{1}$ be one of those intervals. It is mapped by $F\circ\sigma$ diffeomorphically onto a smooth curve $c_0$ of cleave points contained in $L_{a}$. Points of the form $(t,a_{2},\dots,a_{n-1},s) $ for $t<\tilde{t}(a_{2},\dots,a_{n-1},s)$ map through $F$ to a half open ball in $E$. There must be points of $D$ mapping to the other side of $c_0$.
Because of all the above, $c_0$ is also the image of other points in $[s_{1},s_{2}]$. As $c$ is made of cleave points, it must be the image of another component of $C$, which we call $B_{1}$, also mapping diffeomorphically onto $c_0$. Choose a new component $A_{2}$, which is matched to another component $B_{2}$, different from the above, and so on, till the $A_{i}$ and $B_{i}$ are all the components of $C$.

We can write the integral on $B_{i}$ as an integral on $A_{i}$ (we add a minus sign, because the curve is traversed in opposite directions):
$$
\int_{A_{i}} d\tilde{u}(\sigma')+
\int_{B_{i}} d\tilde{u}(\sigma')=
\int_{A_{i}} d\tilde{u}_{l}(\sigma') - d\tilde{u}_{r}(\sigma')
$$
where $d\tilde{u}_{l} $ and $d\tilde{u}_{r} $ are the values of $d\tilde{u} $ computed from both sides. The balanced condition implies
$\sigma'\in \ker(d\tilde{u}_{l} - d\tilde{u}_{r}) $,
and thus the above integral vanishes.
The integral (\ref{u dere- u zurda es una integral}) is absolutely convergent by Lemma \ref{rho is lipschitz}, and the proof follows.

\end{proof}

\paragraph{Remark.} The above proof took some inspiration from \cite[5.2]{Hebda}. The reader may be interested in James Hebda's \emph{tree-like curves}.

\subsection{Structure of S near cleave and crossing points}

In this section we prove some more results about the structure of a balanced split locus near degenerate cleave and crossing points. Besides their importance for proving that $\partial T=0$, we believe they are interesting in their own sake.

\begin{lem}\label{structure of cleave points} 
Let $p\in S$ be a (possibly degenerate) \emph{cleave} point, and let $Q_p=\{x^1,x^2\}$.

There are disjoint univocal neighborhoods $O_1$ and $O_2$ of $x^1$ and $x^2$, and a neighborhood $U$ of $p$ such that for any $q\in U$, $Q_q$ is contained in $O_1\cup O_2$.

Furthermore, if we define:
$$
A_i=\left\lbrace q\in U \text{ such that } Q_q\cap O_i\neq \emptyset  \right\rbrace 
$$
for $i=1,2$, then $A_1\cap A_2$ is the graph of a Lipschitz function, for adequate coordinates in $U$.
\end{lem}
\begin{proof}
The points $x^1$ and $x^2$ are at most of first order, so we can take univocal neighborhoods $O_{1}$ and $O_{2}$ of $x^1$ and $x^2$.
By definition of $Q_p$ and the compactness of $\Omega$, we can achieve the first property, reducing $U$ if necessary.

% The points $x^1$ and $x^2$ are different, but map to the same point. By equation \ref{(F,dF(r)) is an embedding}, we know $\widehat{d_{x^1}F(r)}$ is different from $\widehat{d_{x^2}F(r)}$.
We know $\widehat{d_{x^1}F(r)}$ is different from $\widehat{d_{x^2}F(r)}$.
For fixed arbitrary coordinates in $U$, we can assume that $\{\widehat{d_{x}F(r)} \text{ for } x \in O_1\}$ can be separated by a hyperplane from $\{\widehat{d_{x}F(r)} \text{ for } x \in O_2\}$, after reducing $U$, $O_1$, $O_2$ if necessary.
Therefore, there is a vector $Z_0\in T_p\Omega$ and a number $\delta>0$ such that
\begin{equation}\label{separation between A_1 and A_2} 
\widehat{d_{x}F(r)}(Z) < \widehat{d_{x'}F(r)}(Z)+\delta\qquad \forall\, x\in O_1, \, x'\in O_2
\end{equation} 
% $$
% \widehat{d_{x'}F(r)}(Z_2) > \widehat{d_{x}F(r)}(Z_2)+\varepsilon\; \forall x\in O_1,x'\in O_2
% $$
for any unit vector $Z$ in a neighborhood $G$ of $Z_0$. Let $C^+=\{tZ:t>0,Z\in G\} $ be a one-sided cone containing $Z$. We write $q+C^+$ for the cone displaced to have a vertex in $q$.

Choose $q\in A_1\cap A_2$, and $Z\in G $. Let $\mathcal{R}=\{q'\in U:q'=q+tZ,t>0\} $ be a ray contained in $(q+C^+)\cap U$.
We claim $\mathcal{R}\subset A_1 \setminus A_2 $.

%We identify $\mathcal{R}$ with a subset of the positive real numbers by the linear transformation $t\rightarrow q+tX $ so that $q$ corresponds to $0$. 
For two points $q_1=q+t_1 Z,q_2=q+t_2 Z\in \mathcal{R} $, we say $q_1<q_2$ if and only $t_1<t_2$.
If $\mathcal{R}\cap A_2\neq\emptyset $, let $q_0$ be the infimum of all points $p>0$ in $\mathcal{R}\cap A_2$, for the above order in $\mathcal{R}$.
If $q_0 \in A_1$ (whether $q_0=q$ or not), we can approach $q_0$ with a sequence of points $q_n=F(x_n)>q_0$ carrying vectors $d_{x_n}F(r)$ with $x_n\in O_2$. The limit point of this sequence is $q_0$, and the limit vector is $d_{x}F(r)$ for some $x\in O_2$, but the incoming vector is in $-G$, which contradicts the balanced condition by (\ref{separation between A_1 and A_2}).

If $q_0\in A_2\setminus A_1$, then approaching $q_0$ with points $q<q_n=F(x_n)<q_0$, we get a new contradiction with the balanced property. The only possibility is  $\mathcal{R}\subset A_1\setminus A_2$. As the vector $Z$ is arbitrary, we have indeed $(q+C^+)\cap U\subset A_1\setminus A_2$.

Fix coordinates in $U$, and let $\varepsilon=\frac{1}{2}dist(p,\partial U)$.
Let $B_\varepsilon$ be the ball of radius $\varepsilon $ centered at $p$.
% 
% Let $U_0\subset U$ be an open neighborhood of $p$ whose closure is contained in $U$.
By the above, the hypothesis of lemma \ref{the set and the cone} are satisfied, for $A=A_1\setminus A_2$, the cone $C^+$, the number $\varepsilon$, and $V=B_\varepsilon$.
Thus, we learn from lemma \ref{the set and the cone} that  $A_1\cap A_2\cap B_\varepsilon$ is the graph of a Lipschitz function along the direction $Z_0$ from any hyperplane transversal to $Z_0$.

\end{proof}

The following three lemmas contain more detailed information about the structure of a balanced split locus near a crossing point. The following is stated for the case $n>2$, but it holds too if $n=2$, though then $L$ reduces to a single point $\{a\}$.

\begin{dfn}
 The \emph{normal} to a subset $X\subset T_p^{\ast}  \Omega$ is the set of vectors $Z$ in $T_p \Omega$ such that $\omega(Z)$ is the same number for all $\omega \in X$.
\end{dfn}

\begin{lem}\label{structure of crossing points1} 
Let $p\in S$ be a crossing point. Let $B\subset T_p^{\ast}  \Omega$ be the affine plane spanned by $R_p^\ast $. Let $L$ be the normal to $B$, which by hypothesis is a linear space of dimension $n-2$, and let $C$ be a (double-sided) cone of small amplitude around $L$.

There are disjoint univocal open sets $O_1,\dots,O_N\subset V$ and an open neighborhood $U$ of $p$
%  a constant $c>0$ and numbers $\varepsilon_1, \varepsilon_2 > 0$ 
such that $Q_q\subset \cup_i O_i$ for all $q$ in $U$.

Furthermore, define sets $A_i$ as in lemma \ref{structure of cleave points}, and call $\mathcal{S}=\cup_{i,j}A_i\cap A_j$ the \emph{essential part} of $S$.
Define $\Sigma= \cup_{i,j,k}A_i\cap A_j\cap A_k$ and let $\mathcal{C}= \mathcal{S}\setminus \Sigma$.
\begin{enumerate}
% \item The approximate tangent to $\Sigma$ is contained in $C$ at every $q\in \Sigma$.
\item[(1)] At every $q\in \Sigma$, there is $\varepsilon>0$ such that $\Sigma\cap (q+ B_\varepsilon)\subset q+C$.
\item[(2)] $\Sigma$ itself is contained in $p+C$.
\end{enumerate}
\end{lem}

The next lemma describes the intersection of $\mathcal{S}$ with $2$-planes transversal to $L$.

\begin{lem}\label{structure of crossing points2} 
Let $p\in S$ be a crossing point as above. Let $P\subset  T_p\Omega$ be a $2$-plane intersecting $C$ only at the origin, and let $P_a=P+a$ be a $2$-plane parallel to $P$ for $a\in L $.
\begin{enumerate}
\item If $\vert a\vert<\varepsilon_1$, the intersection of $\mathcal{S} $, the plane $P_a$, and $U$ is a connected Lipschitz tree.
\item The intersection of $\mathcal{S} $, the plane $P_a$, and the annulus of inner radius $c\cdot\vert a\vert $ and outer radius $\varepsilon_2 $:
$$
A(c\:\vert a\vert,\varepsilon_2)=\{q\in U: c\:\vert a\vert<\vert q\vert<\varepsilon_2\}
$$
is the union of $N$ Lipschitz arcs separating the sets $A_i$.
\end{enumerate}
\end{lem}

\begin{figure}[ht]
 \centering
 \includegraphics{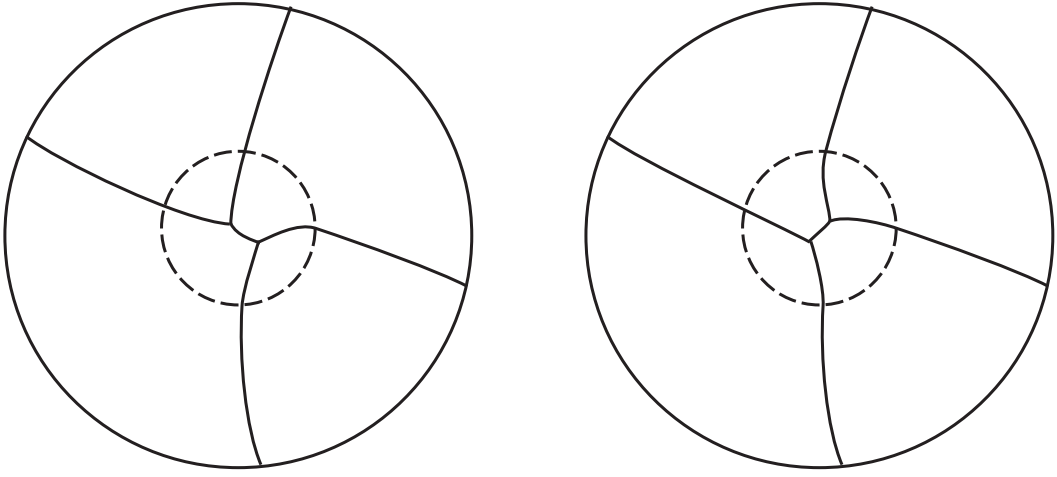}
 % intersections_with_plane_near_crossing.eps: 0x0 pixel, 0dpi, 0.00x0.00 cm, bb=
 \caption{Two possible intersections of a plane $P_a$ with $\mathcal{S}$}
 \label{fig:intersections of plane with S near crossing}
\end{figure}

\paragraph{Remark.} We cannot say much about what happens inside $P_a\cap B(P,c\:\vert a\vert) $. The segments in $ P_a\cap A(c\:\vert a\vert,\varepsilon_2)$ must meet together, yielding a connected tree, but this can happen in several different ways (see figure \ref{fig:intersections of plane with S near crossing}).

Finally, we can describe the connected components of $\mathcal{C}=\mathcal{S}\setminus \Sigma$ within $U$:
\begin{lem}\label{structure of crossing points3} 
Under the same hypothesis, for every $i=1,\dots,N$ there is a coordinate system in $U$ such that:
\begin{itemize}
\item The set $\partial A_i$ is the graph of a Lipschitz function $h_i$, its domain delimited by two Lipschitz functions $f_l$ and $f_r$, for $L^\ast \subset L$:
$$
\partial A_i = \left\lbrace (a,t,h_i(t)), a\in L^\ast, f_{l}(a)<t<f_r(a) \right\rbrace
$$
\item A connected component $\mathcal{C}_0$ of $\mathcal{C}$ contained in $\partial A_i$ admits the following expression, for Lipschitz functions $f_1$ and $f_2$, for $L_0 \subset L$:
$$
\mathcal{C}_0 = \left\lbrace (a,t,h_i(t)), a\in L_0, f_1(a)<t<f_2(a) \right\rbrace
$$
\end{itemize}
\end{lem}
\begin{cor}
$\mathcal{H}^{n-2}(\Sigma)<\infty$.
\end{cor}
\begin{proof}[Proof of corollary]
We apply the \emph{general area-coarea formula} (see \cite[3.2.22]{Federer}), with $W=\Sigma$, $Z=L$, and $f$ the projection from $U$ onto $L$ parallel to $P$, and $m=\mu=\nu=n-2$, to learn:
$$
\int_{\Sigma}ap\: Jf d\mathcal{H}^{n-2}
=\int_L \mathcal{H}^0(f^{-1}(\{z\})) d\mathcal{H}^{n-2}(z)
=\int_L \mathcal{H}^0(\Sigma\cap P_a) d\mathcal{H}^{n-2}(a)
$$
% We know from \ref{structure of crossing points1} that at any point $q\in\Sigma $ there is $\varepsilon$ such that $\Sigma\cap (q+ B_\varepsilon)\subset q+C$ for a fixed cone $C$.
$ap\: Jf|_{\Sigma}$ is bounded from below, so
if we can bound $\mathcal{H}^0(\Sigma\cap P_a) $ uniformly, we get a bound for $\mathcal{H}^{n-2}(\Sigma) $.
% $\Sigma$ is contained in the union of the graphs of the Lipschitz functions $f_1$ and $f_2$ coresponding to the connected components of $\mathcal{C}$. The domain of each such function is a subset of the space $L$ of dimension $n-2$, and there is a countable amount of such connected components. Thus we can bound $\mathcal{H}^{n-2}(\Sigma)$ by $K\int_{a\in L_0}\mathcal{H}^0(\Sigma \cup P_a)$, where $K$ depends on the amplitude of the cone $C$ in lemma \ref{structure of crossing points1}. 
% We just have to show that the number of vertices of the tree $\mathcal{C} \cup P_a $ is finite.

The set $\mathcal{C} \cap P_a \cap U$ is a simplicial complex of dimension $1$, and a standard result in homology theory states that the number of edges minus the number of vertices is the same as the difference between the homology numbers of the complex: $h^1-h^0 $. The graph is connected and simply connected, so this last number is $-1$. The vertices of $\mathcal{C} \cap P_a \cap U $ consist of $N$ vertices of degree $1$ lying at $\partial U$ and the interior vertices having degree at least $3$. 
The \emph{handshaking lemma} states that the sum of the degrees of the vertices of a graph is twice the number of edges, so we get the inequality $2e\geq N+3 \bar{v} $ for the number $e$ of edges and the number $\bar{v}$ of interior vertices. Adding this to the previous equality $e-(N+\bar{v})=-1 $, we get $\bar{v}\leq N-2 $.
We have thus bounded $\bar{v}= \mathcal{H}^0(\Sigma\cap P_a)$ with a bound valid for all $a$.

% We provide a graph-theoretic proof. Consider $\mathcal{C} \cup P_a \cup U$ as a graph in the plane $P_a$ partitioning the plane into a single open set, with $N$ vertices of multiplicity $1$ and the interior vertices having multiplicity at least $3$. Recall the Euler formula:
% $f-e+v=2 $ ($f$ is faces, $e$ is edges, $v=N+\bar{v}$ is vertices, separating the tips of the tree from the interior vertices) and the inequality $2e\leq N+3 \bar{v} $ which is equivalent to the statement that interior vertices have multiplicity at least $3$. Combining both we get a bound for the number of interior vertices: $v\leq N-2 $.
\end{proof}

\begin{proof}[Proof of \ref{structure of crossing points1}]
This lemma can be proven in a way similar to \ref{structure of cleave points}, but we will take some extra steps to help us with the proof of the other lemmas.

First, recall the map $\Delta$ defined in (\ref{definition of the embedding of V into TastOmega}).
Each point $x$ in $\Delta^{-1}(R_p^{\ast})$ has a univocal neighborhood $\mathcal{O}_x $. 
Recall $R_p^{\ast}$ consists only of covectors of norm $1$. Let $\gamma $ be the curve obtained as intersection of $B$ and the covectors of norm $1$.
Instead of taking the neighborhoods $\mathcal{O}_x $ right away, which would be sufficient for this lemma, we cover $R_p^{\ast}$ with open sets of the form $\Delta(\mathcal{O}_x) \cap \gamma$.

By standard results in topology, we can extract a finite refinement of the covering of $R_p^{\ast}\subset\gamma$ by the sets $\Delta(\mathcal{O}_x)\cap \gamma$ consisting of disjoint non-empty intervals $I_1,\dots,I_N$.
Let $\tilde{I}_i$ be the set of points $tx$ for $t\in (1-\varepsilon_1,1+\varepsilon_1)$ and $x\in I_i$, and choose a linear space $M_0$ of dimension $n-2$ transversal to $B$. Define the sets of our covering:
$$
O_i = \Delta^{-1}(\tilde{I}_i + B(M_0,\varepsilon_2))
$$
for the ball of radius $\varepsilon_2$ in $M_0$ ($\varepsilon_1 $ and $\varepsilon_2 $ are arbitrary, and small).

We can assume that $Q_q\subset \cup_i O_i$ for all $q$ in $U$ by reducing $U$ and the $O_i$ further if necessary, hence we only need to prove the two extra properties to conclude the theorem. 

The approximate tangent to $\Sigma$ at a point $q\in \Sigma\cap U$ is contained in the normal to $R_q^{\ast}$ (see the definition of approximate tangent in \cite{Nosotros} and use the proof of proposition 7.3 there, or merely use the balanced property).
If $R_q^{\ast}$ is contained in a sufficiently small neighborhood of $\gamma$ and contains points from at least three different $I_i$, its normal must be close to $L$.
Thus if we chose $\varepsilon_1$ and $\varepsilon_2$ small enough, the approximate tangent to $\Sigma\cap U$ at a point $q\in \Sigma$ is contained in $C$.
If property (1) did not hold for any $\varepsilon$ at a point $q$, we could find a sequence of points converging to $q$ whose directions from $q$ would remain outside $C$,  violating the above property.

Finally, the second property holds if we replace $U$ by $U\cap B_\varepsilon$, for the number $\varepsilon$ that appears when we apply property (1) to $p$.
\end{proof}

\begin{figure}[ht]
 \centering
 \includegraphics{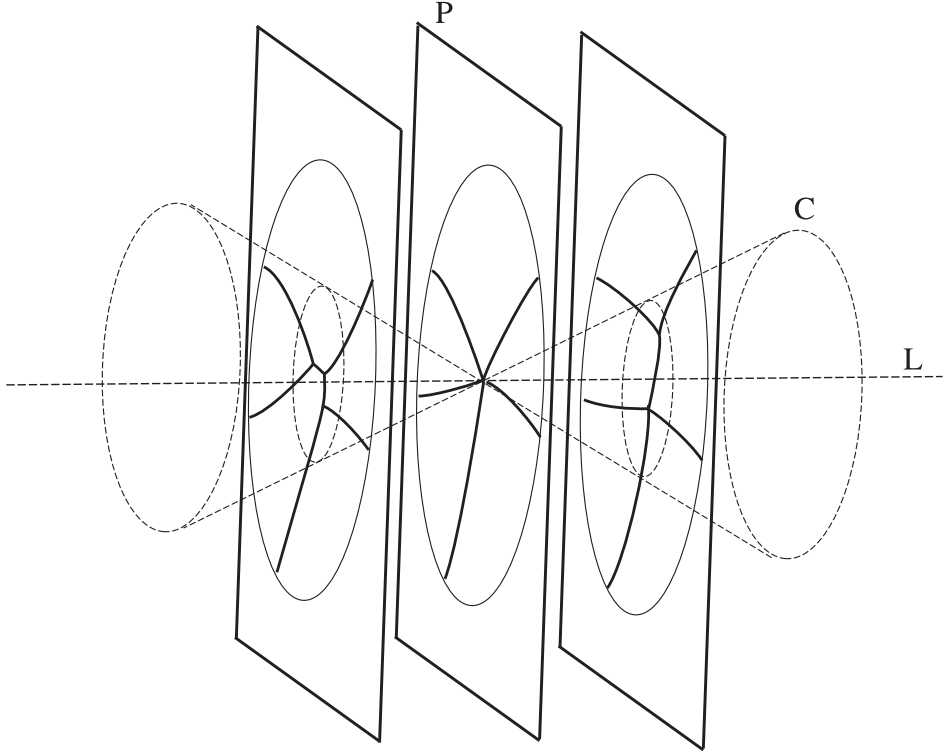}
 % near_a_crossing_point.eps: 0x0 pixel, 0dpi, 0.00x0.00 cm, bb=
 \caption{$\mathcal{S}$ near a crossing point} 
\label{fig: S near a crossing point}
\end{figure}

\begin{proof}[Proof of \ref{structure of crossing points2}]
Just like in \ref{structure of cleave points}, we can assume that each set $\{\widehat{d_{x}F(r)} \text{ for } x \in O_i\}$ can be separated from the others by a hyperplane (e.g., a direction $Z_i$), such that:
\begin{equation}\label{separation between A_i and the others} 
\widehat{d_{x}F(r)}(Z) < \widehat{d_{x'}F(r)}(Z)+\delta\qquad \forall \, x\in O_i,\, x'\in O_j,\, i\neq j
\end{equation} 
for some $\delta>0$ and any unit vector $Z$ in a neighborhood $G_i$ of $Z_i$.
Thanks to the care we took in the proof of the previous lemma, we can assume all $Z_i$ belong to the plane $P$ in the statement of this lemma: indeed the intervals $I_i$ can be separated by vectors in any plane transversal to $L$, and the sets $\Delta(O_i)$ are contained in neighborhoods of the $I_i$.

Define the one-sided cones $C^+_i=\{tZ:t>0,Z\in G_i\}$.
The above implies that the intersection of each $C_i^+$ with $P$ is a nontrivial cone in $P$ that consists of rays from the vertex. 

By the same arguments in \ref{structure of cleave points}, we can be sure that whenever $q\in A_i $, then $(q+C^+_i )\cap U\subset A_i$. 
This implies that $\partial A_i$ is the graph of a Lipschitz function along the direction $Z_i$ from any hyperplane transversal to $Z_i$. 
% Let $H_i$ be one such hyperplane that contains the subspace $L$.
We notice $\partial A_i$ is (Lipschitz) transversal to $P$, so for any $a\in L$, $\partial A_i\cap P$ is a Lipschitz curve.
As the cone $C$ is transversal to $P$, and the tangent to $\Sigma$ is contained in $C$, we see $\Sigma\cap P_a$ consists of isolated points. 

Thus $\mathcal{S} \cap P_a$ is a Lipschitz graph and $\Sigma\cap P_a$ is the set of its vertices.
% It is a tree, because $P_a \cap \mathcal{S}$ is a deformation retract of $P_a\cap U$, where a point of $A_i$ retracts to $\mathcal{S}$ following the direction $-Z_i$ (es cierto, pero no obvio).
If it were not a tree, there would be a bounded open subset of $P_a \cap U\setminus\mathcal{S}$ with boundary contained in $\mathcal{S}$. An interior point $q$ belongs to some $A_k$. Then the cone $q+C^+_k$ is contained in $A_k$, but on the other hand its intersection with $P_a$ contains a ray that must necessarily intersect $\mathcal{S}$, which is a contradiction.

We notice $P_a\cap (p+C^+_i)\subset A_i$.
This set is a cone in $P_a$ (e.g. a circular sector) with vertex at most a distance $c_1\vert a\vert $ from $p+a$, where $c_1>0$ depends on the amplitude of the different $C_i$.

If $a=0$, the $N$ segments departing from $p$ with speeds $Z_i$ belong to each $A_i$ respectively.
Let us assume that the intervals $I_i$ appearing in the last proof are met in the usual order $I_1,I_2\dots,I_N$ when we run along $\gamma$ following a particular orientation, and call $P^i$ the region delimited by the rays from $p$ with speeds $Z_i$ and $Z_{i+1}$ (read $Z_1$ instead of $Z_{N+1}$).%, and the circumference of radius $\varepsilon_2$.
If there is a point $q\in P^i\cap A_k\cap B(\varepsilon_2)$ for sufficiently small $\varepsilon_2$, then $(q+C^+_k)\cap U$ would intersect either $p+C^+_i$ or $p+C^+_{i+1}$, and yield a contradiction if $k$ is not $i$ or $i+1$.
Thus $P^i\subset A_i\cup A_{i+1}$. Clearly there must be some point $q$ in $P_i\cap A_i\cap A_{i+1}$, to which we can apply lemma \ref{structure of cleave points}. 
$A_i\cap A_{i+1}$ is a Lipschitz curve near $q$ transversal to $Z_i$ (and to $Z_{i+1}$), and it cannot turn back.
The curve does not meet $\Sigma$, and it cannot intersect the rays from $p$ with speeds $Z_i$ and $Z_{i+1}$, so it must continue up to $p$ itself.
For any $q\in A_i\cap A_{i+1}$, the cone $q+C^+_i$ is contained in $A_i$, and the cone $q+C^+_{i+1}$ is contained in $A_{i+1}$. This implies there cannot be any other branch of $A_i\cap A_{i+1}$ inside $P_i$.

This is all we need to describe $\mathcal{S}\cap P \cap B(\varepsilon_2)$: it consists of $N$ Lipschitz segments starting at $p$ and finishing in $P \cap \partial B(\varepsilon_2) $. The only multiple point is $p$.

For small positive $\vert a\vert$, we know by condition (2) of the previous lemma that $P_a\cap \Sigma\subset C\cap P_a= B(c_2\vert a\vert) \cap P_a$ for some $c_2>0$.
Similarly as above, define regions $P_a^i\subset P_a \cap A(c\vert a\vert,\varepsilon_2)$ delimited by the rays from $a$ with directions $Z_i$ and $Z_{i+1}$, and the boundary of the ring $A(c\vert a\vert,\varepsilon_2)$, for a constant $c>\max(c_1,c_2)$.
Take $c$ big enough so that  for any $q\in P_a^i $ and any $k\neq i,i+1$ , $q+C^+_k \cap U\cap P_a$ intersects either $p+C^+_i$ or $p+C^+_{i+1} $.
The same argument as above shows that $A_i\cap A_{i+1} \subset P_a^i \subset A_i\cup A_{i+1} $.
We conclude there must be a Lipschitz curve of points of $A_i\cap A_{i+1}$, which starts in the inner boundary of $A(c\vert a\vert,\varepsilon_2)$, and ends up in the outer boundary.

\end{proof}

\begin{figure}[ht]
 \centering
 \includegraphics{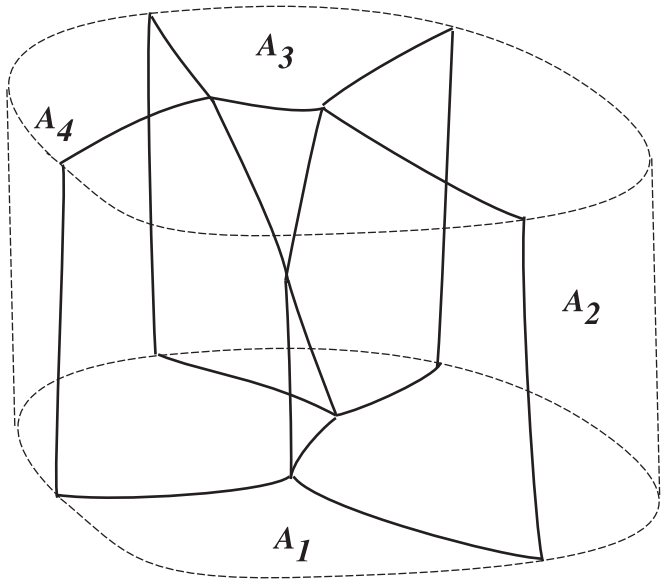}
 % crossing_point_v2.eps: 0x0 pixel, 0dpi, 0.00x0.00 cm, bb=
 \caption{A neighborhood of a crossing point (this view is rotated with respect to figure \ref{fig: S near a crossing point})}
 \label{fig: neighborhood of a crossing point}
\end{figure}

\begin{proof}[Proof of \ref{structure of crossing points3}]
First we assume $U$ has a product form $U=L^\ast\times P^\ast$ for open discs $L^\ast\subset L$ and $P^\ast=B(P,\varepsilon_2)\subset P$.

Recall $\partial A_i$ is the graph of a Lipschitz function along the direction $Z_i$ from any hyperplane transversal to $Z_i$.
Let $H_i=L+W$ be one such hyperplane that contains the subspace $L$ and the vector line $W\subset P$, and construct coordinates $L\times W\times <Z_i>$.
It follows from the previous lemma that $\partial A_i\cap P^\ast_a$ is a connected Lipschitz curve.
In these coordinates $\partial A_i$ is the graph of a Lipschitz function $h_i$.
Its domain, for fixed $a$, is a connected interval, delimited by two functions $f_l:L^\ast\rightarrow W$ and $f_r:L^\ast\rightarrow W$. Condition (1) of lemma \ref{structure of crossing points1} assures they are Lipschitz.

A connected component $\mathcal{C}_0$ of $\mathcal{C}$ is contained in only one $A_i\cap A_j$.
We can express it in the coordinates defined above for $\partial A_i$.
The intersection of $\mathcal{C}_0$ with each plane $P_a$ is either empty or a connected Lipschitz curve. The second part follows as before.
\end{proof}

\subsection{Conclusion}

Using lemma \ref{uniqueness near order 1 points}, we show without much effort that $\partial T$ vanishes near edge points.
Using the structure results from the previous section, we show also that it vanishes at cleave points (including degenerate ones) and crossing points.

\begin{prop}
 Let $p\in S$ be an \emph{edge} point. 
Then the boundary of $T$ vanishes near $p$.
\end{prop}
\begin{proof}
Let $p$ be an edge point with $Q_p=\{x\}$.
Let $O$ be a univocal neighborhood of $x$.
It follows by a contradiction argument that there is an open neighborhood $U$ of $p$ such that $Q_q\subset O$ for all $q\in U $.
Recall the definition of $T$:
$$
T(\phi)=\sum_{j}\int_{\mathcal{C}_{j,1}}(h_{1}-h_{2})\phi
$$
For any cleave point $q\in U$ with $Q_q=\{x_1,x_2\}$, $h_i(q)=\tilde{u}(x_i)$.
By the above, both $x_1$ and $x_2$ are in $O$.
As $O$ is univocal, we see $h_1=h_2$ at $q$. The integrand of $T$ vanishes near $p$, and thus $\partial T =0$.
\end{proof}

\begin{prop}
Let $p\in S$ be a (possibly degenerate) \emph{cleave} point. 
Then $\partial T$ vanishes near $p$.
\end{prop}
\begin{proof}
% Let $Q_{p}=\lbrace x_{1},x_{2} \rbrace$. Let $O_{1}$ and $O_{2}$ be two neighborhoods of $x_{1}$ and $x_{2}$ respectively such that either $F\vert_{O_{i}}$ is non-singular, or $O_{i}$ is obtained from lemma \ref{uniqueness near order 1 points}.
% By reducing $O_1$, $O_2$ and $U$ further if necessary, we can assume that both $O_{1}$ and $O_{2}$ map into an open set $U$ such that for any point $q\in U$, $Q_{q}\subset O_{1}\cup O_{2}$.
% 
% If the cleave point $q$ belongs to the component $C_{j}$ and the two points in $Q_{q}$ come from the same $O_{i}$, then by lemma \ref{uniqueness near order 1 points} the two values of $\tilde{u}$ are the same, and the component $C_{j}$ does not contribute to the integral in \ref{definition of T}.
% Define $A_{i}$ as the set of points $q\in U$ such that $Q_{q}$ contains a point of $O_{i}$. Then we can ignore the components of $S$ that are not contained in $A_1\cap A_2$. 
% $A_1\cap A_2$ is a Lipschitz hypersurface in $U$ of finite $\Hnuno$-measure.
% % , as follows from \ref{the set and the cone} by another familiar argument. %(esto sobra??, pq ya sabemos que C es Lipschitz)

Use the sets $U$, $A_1$ and $A_2$ of lemma \ref{structure of cleave points}.

Whenever $\phi$ is a $n-1$ differential form with support contained in $U$, we can compute:
$$
T(\phi)=\int_{A_1\cap A_2}(h_1-h_2)\phi
$$
The components of cleave points inside either $A_1$ or $A_2$ do not contribute to the integral, for the same reasons as in the previous lemma. Recall the definition of $\partial T$, for a differential $n-2$ form $\sigma$:
$$
\partial T(\sigma)=T(d\sigma)=
\int_{A_1\cap A_2}(h_1-h_2)d\sigma
$$
We can apply a version of Stokes theorem that allows for Lipschitz functions. We will provide references for this later:

%%%%%%%%%%%%%%%%%%%%%%%%
%Gauss-Green theorem could be applied instead, using the lipeomorphism between A_1\cap A_2 and a subset of the boundary. The usual statement of Gauss-Green requires C^1 functions, but there are several statements around that cover Lipschitz functions (e.g. MR2336310: The divergence theorem for unbounded vector fields)
%%%%%%%%%%%%%%%%%%%%%%%%
$$
T(d\sigma)=\int_{A_1\cap A_2}d(h_1-h_2)\sigma
$$

The balanced condition imposes that for any vector $v$ tangent to $A_1\cap A_2$ at a non-degenerate cleave point $q$ with $Q_q=\{x_1,x_2\}$.
$$
\hat{X}^1(v)=\hat{X}^2(v)
$$
for the incoming vectors $X^i=d_{x_i}F(r)$. Recall that $\Hnuno$-almost all points are cleave, and $dh_i$ is dual to the incoming vector $X^i$, so $T(d\sigma)=0 $.
\end{proof}

% ??\begin{dfn}\label{local solution}
% Let $(p,tV)$ be a vector 
% %of order $1$ 
% in \ref{domain of exp}, and let $U$ be a neighborhood of $(p,tV)$ in the set \ref{domain of exp}.
% 
% The \emph{local solution} from $U$ to (\ref{HJequation}) and (\ref{HJboundarydata}) is the function:
% \begin{equation}
% u(x)=\inf_{y\in U}
% \left\lbrace 
%    d(x,y)+ g(y)
% \right\rbrace
% \end{equation} 
% \end{dfn}??

\begin{prop}
Let $p\in S$ be a \emph{crossing} point. 
Then the boundary of the current $T$ (defined in \ref{definition of T}) vanishes near $p$.
\end{prop}

\begin{proof}

We use lemma \ref{structure of crossing points3} to describe the structure of connected components of $\mathcal{C}$ near $p$.
Let $\Sigma_T$, the set of \emph{higher order points}, be the set of those points such that $R_q^{\ast}$ spans an affine subspace of $T^{\ast}_q \Omega$ of dimension greater than $2$.

Take any connected component $\mathcal{C}_0$ of $\mathcal{C}$ contained in $\partial A_i$.
$\partial \mathcal{C}_0 $ decomposes into several parts: 

\begin{itemize}
 \item The regular boundary, consisting of two parts $D_1$ and $D_2$:
$$D_1=\{(a_1,\dots,a_{n-2},f_1(a),h_i(f_1(a))), \forall a\in L^\ast \text{ such that } f_l(a)<f_1(a)<f_2(a)\}
$$ 
$$
D_2=\{(a_1,\dots,a_{n-2},f_2(a),h_i(f_2(a))), \forall a\in L^\ast \text{ such that } f_1(a)<f_2(a)<f_r(a)\}
$$

\item The points of higher order, or $\partial \mathcal{C}_0 \cap \Sigma_T$.

\item The singular boundary, or those points $q=(a_1,\dots,a_{n-2},f_1(a),h_i(f_1(a))) $ where $f_1(a)=f_2(a)$ and $R_q$ is contained in an affine plane.

\item A subset of $\partial U$.

\end{itemize}

Using a version of Stokes theorem that allows for Lipschitz functions, we see that
$$
\int_{\mathcal{C}_0} v d\sigma= \int_{\mathcal{C}_0} d(v \sigma)-\int_{\mathcal{C}_0} (dv) \sigma=
\int_{D_1} v \sigma - \int_{D_2} v \sigma-\int_{\mathcal{C}_0} (dv) \sigma
$$
for any function $v$ and $n-2$ form $\sigma$ with compact support inside $U$.
Indeed, the last coordinate of the parametrization of $\mathcal{C}_0 $ is given by a Lipschitz function, so we can rewrite the integral as one over a subset of $L\times W$, and only Gauss-Green theorem is needed.
We can apply the version in \cite[4.5.5]{Federer}, whose only hypothesis is that the current $\Hnuno\lfloor \partial \mathcal{C}_0 $ must be representable by integration.
Using \cite[4.5.15]{Federer} we find that it is indeed, because its support is contained in a rectifiable set 
Here we are assuming that $D_1$ is oriented as the boundary of $\mathcal{C}_0 $, while $D_2$ is oriented in the opposite way, to match the orientation of $D_1$.
%and a set of null Hndos measure
% This last part of $\partial C_0$ is also $n-2$ rectifiable as it can be expressed as a Lipschitz transversal intersection of $\partial U$ and a Lipschitz hypersurface, but we actually don't need to deal with that:

Notice we have discarded several parts of $\partial \mathcal{C}_0$:
\begin{itemize}
 \item A subset of $\partial \mathcal{C}_0$ inside $\partial U$ does not contribute to the integral because $supp(\sigma)\subset\subset U $.
 \item $\partial \mathcal{C}_0 \cap \Sigma_T$ does no contribute because it has Hausdorff dimension at most $n-3$.
 \item The singular boundary does not contribute either, because the normal to $\widetilde{\mathcal{C}_0}$ at a point of the singular boundary does not exist (see \cite[4.5.5]{Federer}).
\end{itemize}

% The transversal boundary has zero $\Hndos$ measure and thus can be discarded. 
% At a point in the singular boundary the two functions $f_1$ and $f_2$ approach with the same derivative. Thus if we restrict the integral to the left to the set $\{a: \vert f_1(a)-f_2(a)\vert<h\} $, we only lose a term of order $\varepsilon h$ for $\varepsilon$ arbitrary small. Then we can apply the Stokes theorem for a Lipschitz domain, and we get an integral over a surface of order

% *For some components $C_0$, part of $\partial C_0$ may actually lie in $\partial U$. In this case, as $supp(\sigma)\subset\subset U $, the corresponding integral vanishes.

% *We now prove that, at any point of $\Sigma \setminus \Sigma_T$, we have $\partial T=0$, which completes the proof of this lemma.

We now prove that $\partial T=0$.

For a form $\sigma $ of dimension $n-2$ and compact support inside $U$:
$$
T(d\sigma)=\sum_{i}\int_{\mathcal{C}_i}(h_l-h_r)d\sigma=
\sum_{i}\int_{\mathcal{C}_i}d(h_l-h_r)\sigma+\sum_{i}\left(\int_{D_{i,1}}(h_l-h_r)\sigma-\int_{D_{i,2}}(h_l-h_r)\sigma\right)
$$
where $D_{i,1}$ and $D_{i,2}$ are the two parts of the regular boundary of $\mathcal{C}_i $.

%%%%%%%%%%%%%%%%%%%%%%%%
%Again, Gauss-Green theorem could be applied instead, using the lipeomorphism between A_1\cap A_2 and a subset of the boundary. The usual statement of Gauss-Green requires C^1 functions, but there are several statements around that cover Lipschitz functions (e.g. MR2336310: The divergence theorem for unbounded vector fields) 
% Probablemente lo mejor sea usar la version del Federer, con el vector normal y tal
%%%%%%%%%%%%%%%%%%%%%%%%
The first summand is zero and the remaining terms can be reordered (the sum is absolutely convergent because $h$ is bounded and $\Hndos(\Sigma)$ is finite):
$$
\sum_{i}\left(\int_{D_{i,1}}(h_l-h_r)\sigma-\int_{D_{i,2}}(h_l-h_r)\sigma\right)=
\int_{\Sigma\setminus\Sigma_T} \sum_{(i,j)\in I(q)}(h_{i,j,l}-h_{i,j,r})\sigma dq
$$
where every point $q\in \Sigma\setminus\Sigma_T$ has a set $I(q)$ consisting of those $i$ and $j=1,2$ such that $q$ is in the boundary part $D_j$ of the component $\mathcal{C}_i$.
% This set consists of components $\mathcal{C}_i\subset A_j\cap A_k$ such that $q\in \overline{\mathcal{C}_i} $.
% and there is no index in $R_q$ between $j$ and $k$. (a que venia esto?)
The integrand at point $q$ is then:
$$
\sigma\sum_{(i,j)\in I(q)}(h_{i,j,l}-h_{i,j,r})
$$
where $h_{i,j,l}$ is the value of $\tilde{u}(x)$ coming from the side $l$ of component $\mathcal{C}_i$ and boundary part $D_j$.

By the structure lemma \ref{structure up to codimension 3}, we can restrict the integral to crossing points.
Let $O_1,\dots,O_N$ be the disjoint univocal sets that appear when we apply \ref{structure of crossing points1} to $p$.
For a crossing point $q$, $I(q)$ is in correspondence with the set of indices $k$ such that $O_k\cap Q_p \neq \emptyset $. Indeed, the intersection of $\mathcal{S} $ with the plane $P_a$ containing $q$ is a Lipschitz tree, and $q$ is a vertex, and belongs to the regular boundary of the components that intersect $P_a$ in an edge. 
The $h_{i,j,l}$ in the sum appear in pairs: one is the value from the left coming from one component $\mathcal{C}_i$ and the value from the right of another component $\mathcal{C}_{i'}$.
Each one comes from a different side, so they carry opposite signs, and they cancel.
The integrand at $q$ vanishes altogether, so $\partial T =0$.

% where $K(q)$ is the set of indices $k$ such that $\mathcal{O}_k\cap Q_p \neq \emptyset $ and $h_k$ is the value of $\tilde{u}(x)$ for any $x\in Q_q\cap \mathcal{O}_k $.

\end{proof}

\end{document}